\documentclass[onecolumn,twoside]{IEEEtran}\renewcommand{\baselinestretch}{1.5}
\usepackage{amssymb,amsmath}
\usepackage{graphicx}
\usepackage{cite}
\usepackage{color}

\DeclareMathOperator{\diag}{diag}

\newcommand{\R}{\mathbb{R}}
\newcommand{\yi}{\mathbf{1}}
\newcommand{\N}{\mathcal{N}}
\newcommand{\G}{\mathcal{G}}

\newtheorem{theorem}{Theorem}
\newtheorem{corollary}{Corollary}
\newtheorem{lemma}{Lemma}

\newtheorem{assumption}{Assumption}
\newtheorem{example}{Example}
\newtheorem{remark}{Remark}

\def\endproof{\hspace*{\fill}~\IEEEQEDclosed\par\endtrivlist\unskip}

\title{Robustness Analysis of Asynchronous Sampled-Data Multi-Agent Networks With Time-Varying Delays\thanks{This work was supported by the National Natural Science Foundation of
China under Grant 61422302, the Program for New Century Excellent
Talents in University under Grant NCET-13-0178, and the 111 Project under Grant B17048.}}
\author{Feng~Xiao, {\it Member, IEEE}, ~Yang~Shi, {\it Fellow, IEEE}, ~Wei~Ren, {\it Fellow, IEEE}
\thanks{F. Xiao is with the School of Control and Computer Engineering,
North China Electric Power University, Beijing 102206, China (e-mail:
fengxiao@ualberta.net).}
\thanks{Yang Shi is with the Department of Mechanical
Engineering, University of Victoria,
Victoria, BC,  V8N 3P6, Canada (e-mail: yshi@uvic.ca).}
\thanks{Wei Ren is with the Department of Electrical and Computer Engineering,
University of California, Riverside, CA 92521 USA (email:~{ren@ee.ucr.edu}).}
}

\begin{document}
\date{}
\maketitle

\begin{abstract}
In this paper, we study the simultaneous stability problem of a finite number of locally inter-connected linear subsystems under practical constraints, including asynchronous and aperiodic sampling, time-varying delays, and measurement errors. We establish a new Lyapunov-based stability result for such a decentralized system. This system has a particular simple structure of inter-connections, but it captures some key characteristics of a large class of intermediate models derived from the consensus analysis of multi-agent systems. The stability result is applicable to the estimation of the maximum allowable inter-sampling periods and time delays based on individual dynamics and coupling structures in the scenarios of consensus control via asynchronous sampling of relative states and asynchronous broadcasting of self-sampled states respectively. The asynchrony of aperiodic sampling and the existence of measurement errors allow the utilization of some kinds of quantizing devices, such as Logarithmic quantizers, in the process of data sampling, and allow the introduction of a period of dwell time after each update of state measurement to eliminate the Zeno behavior of events in event-based control. The extension in the case with input saturations and input delays is also discussed.
 \end{abstract}


\section{Introduction}
Due to the attractive advantages in signal processing and transmission, digital devices have found their wide applications in modern control systems\cite{Dorf_book_mcs}. An analog-to-digital converter  is responsible for converting a continuous-time signal into a digital signal by sampling and quantization. A digital controller gathers all input signals and compute an output to achieve the desired control purpose. Sampled-data control deals with such hybrid continuous-time and discrete-time systems and usually assumes that these digital devices share the same clocks and process their data periodically and synchronously\cite{Chen_book}. This ideal assumption can formulate the considered digital systems as standard LTI discrete-time systems by plant discretization and servers as the basis of many fundamental results, for example, in stability and optimal control\cite{Francis_tac1988,Bamieh_tac1992}. When these digital components, such as A/D samplers and  D/A zero-order holds, work at different frequencies or they are spatially scattered in a large area without a central clock, the sampled-data systems become inherently asynchronous\cite{Voulgaris_cdc1993}. Although synchronous sampled-data models may be used to approximate the asynchronous ones in some of such cases\cite{Ritchey_tac1989}, this treatment could sacrifice some dynamic  details in plants.

{\it Literature Review:} Asynchronous sampling has been brought to the attention of researchers in the control community for more than three decades\cite{Sridharan_auto1985}. It has obtained enormous results, and is now developing rapidly with the growth of multi-agent theory as one of the main research areas. Generally, asynchronous sampling can be found in the following three types of systems:

 (1) Multi-rate systems. Multi-rate sampling is one of the earliest motivating examples for studying asynchronous systems, where sampling and hold elements work periodically at different rates with irrational ratios\cite{Ritchey_tac1989,Voulgaris_cdc1993,Fang_acc1994,Sagfors_auto1997,Moarref_auto2014LMIs}. Examples include dual-rate linear systems with a single sampler and a single hold \cite{Voulgaris_cdc1993,Sagfors_auto1997} and multi-rate linear systems with multiple asynchronous samplers and zero-order holds\cite{Ritchey_tac1989,Fang_acc1994,Moarref_auto2014LMIs}. The corresponding control problems include the optimal LQG control\cite{Voulgaris_cdc1993,Sagfors_auto1997} and stability problems\cite{Ritchey_tac1989,Fang_acc1994,Moarref_auto2014LMIs}.

(2) Networked control systems (NCSs). Theoretically, some NCSs can be modeled as multi-rate systems. Due to the long-distance transmission of information, communication networks in NCSs may suffer from  packet loss and time delays\cite{HuLS_auto2007,Wangxf_scl2012,CaoM_auto2014,Tavassoli_tac2014LMIs,Ge_sp2016LMIs} and often have multiple independent information channels\cite{Wangxf_scl2012,Antunes_auto2013,Tavassoli_tac2014LMIs,Ge_sp2016LMIs}. To ensure effective information transmission with reduced costs, asynchronous event-triggered sampling also has been considered in NCSs\cite{Wangxf_scl2012,CaoM_auto2014}

(3) Multi-agent systems. The asynchronous property in signal processing becomes more prominent in multi-agent networks, in which, finite numbers of subsystems, equipped with independent signal-sensing and data-processing devices, are inter-connected to perform some cooperative tasks\cite{Tsitsiklis_tac1986,Lin_cdc2004}. For example, in the formation control of multi-robot systems, each robot should detect the positions and velocities of other adjacent robots for route planning. However, the installed sensors, such as ultrasonic sensors or laser sensors, usually cannot monitor all objects in $360$-degree coverage at a time. Furthermore, they could also be affected by environmental interference. So in such cases, it is practically preferable that each  robot collects the data of its neighbors in some order (not necessarily periodically) according to its own schedule and within its sensing/processing capacities. These kinds of realistic scenarios raise the problem of asynchronous and aperiodic sampling. The sensor scheduling can be time-driven\cite{Cao_tac2008,Gao_ijc2010,Gao_tac2011}, event-driven\cite{Dimarogonas_tac2012}, or mixture (of time and event)-driven\cite{Heemels_tac2013,Xiao_scl2016}.

With asynchronous sampling, even if the involved digital devices operate periodically, sampled-data systems become non-periodic and time-varying \cite{Sagfors_auto1997}. Asynchronous sampling can destroy the stability of systems which are stable in synchronous environments \cite{Fang_acc2005}, and induce additional time-varying delays in system analysis\cite{Freirich_auto2016LMIs}. So the analysis of asynchronous systems is more difficult and challenging than their synchronous counterparts, and the asynchronous mechanism also sets a strict requirement on the robustness of designed controllers or algorithms with respect to aperiodic sensor scheduling, sensing errors and processing/communcation delays. In asynchronous multi-agent systems, most results have been developed based on the simple individual models of  single-integrators\cite{Xiao_scl2016} and double-integrators\cite{Cao_tac2008,Meng_auto2013,Zhan_TCSI2015}, but relatively few studies have been done on the general linear models, with the exception of the event-triggered control\cite{Zhu_auto2014,Hu_ta2016,Yang_auto2016}.

{\it Contributions:} This paper will establish a Lyapunov-based stability result for an asynchronous multi-agent system with the consideration of the above-mentioned  practical issues and then apply it to solving several representative asynchronous consensus problems in one framework. The contributions are summarized as follows:

First, we set up a basic stability model for asynchronous coordination of multi-agent systems and  solve the problem of how to estimate the maximum length of sampling intervals by subsystem matrices and inter-connection structures.  In the model, all subsystems, represented by the general linear state-space models, share the same system matrix and feedback matrix. They communicate with each other via discrete-time signals produced by samplers and zero-order holds. Examples of such a model without sampling can be found in the coordination analysis of a large range of multi-agent systems\cite{Xiao_iet2007,Qin_tcns2015}. We impose  relaxed assumptions on the A/D and D/A devices and information channels, including asynchronous and aperiodic sampling, measurement errors, and time-varying delays. The proposed result well describes the robustness of the asynchronous sampled-data system based on the stability of the continuous-time system without sampling.  In \cite{Freirich_auto2016LMIs}, by a Lyapunov-Krasovskii method and LMIs, the authors studied the protocol design problem of a similar model. But it is different from ours in the following two aspects: (1) in our model,  given any inter-connection structure,  the subsystems are inherently coupled by discrete-time signals;  in \cite{Freirich_auto2016LMIs}, the subsystems are coupled by continuous-time signals and the feedback matrices and scheduling protocols for discrete-time signals need to be designed accordingly; (2) In our model, the information is aperiodically sampled with measurement errors and transmitted with time delays; in \cite{Freirich_auto2016LMIs}, the information is sampled and transmitted according to the proposed Round-Robin (RR) or Try-Once-Discard (TOD) protocol. Note that the sampling with time-varying periods is also referred to as asynchronous sampling in \cite{Seuret_auto2012LMIs,Jiang_wdecns2010LMIs,Omran_adhs12LMI} but it is  a different definition from the ``asynchronous sampling"  in this paper.

Second, we solve the sampled-data consensus problems of the following asynchronous multi-agent systems in one framework: (1) networks of general linear agents with asynchronous sampling of relative states; (2) networks of single-integrators and marginally stable systems with asynchronous  broadcasting of self-sampled states. Feedbacks with relative states and broadcast communication have been widely used in the consensus and formation control of multi-agent systems; particularly, broadcasting is an important way of keeping the state average unchanged; see \cite{saber_tac2004,Cortes_auto2009,Guo_auto2013,Tassiulas_auto1999,Wang_acc2008,Nedic_tac2011,Seyboth_auto2013,Meng_auto2013,Wang_neuro2016} and references therein. Our work on asynchronous broadcasting is partly inspired by this observation and the protocol design method presented for the  event-based control of single-integrator and double-integrator agents in \cite{Seyboth_auto2013}. Due to the challenging difficulties of asynchronous consensus analysis, the existing results on sampled-data consensus mainly focus on synchronous sampling; that is, all data should be sampled at the same time with constant or variant frequencies. The involved individual models are usually single-integrators\cite{Xie_acc2009}, double-integrators\cite{Zhang_tac2010,Cao_ijc2010,Qin_tac2012}, or second-order oscillators\cite{ZhangH_scl2012}. There are relatively rare reports on asynchronous sampled-data consensus of general linear agents and they mostly deal with event-triggered consensus \cite{Zhu_auto2014,Hu_ta2016,Yang_auto2016}, which is intrinsically different from the time-driven style of asynchrony in this paper.  Our work is also different from the previous results based on dynamic outputs of  controllers \cite{Cao_tac2008,Zhan_TCSI2015}.

Finally, taking advantages of asynchronous sampling and allowable measurement errors, we can extend the obtained results  further to deal with some of the general cases with quantization in discretizing continuous-time signals, and we rigorously  prove the effectiveness of the method for eliminating the Zeno behavior by introducing a dwell time after each measurement in event-based control. To the best of our knowledge, the related problems on general linear multi-agent systems have not been studied in the existing literature. This extension is illustrated by the implementation of Logarithmic quantizers \cite{Aldajani_dsp2008} and  event-triggering conditions based on state errors \cite{Wang_acc2008,Heemels_tac2013}. We also show the application of the results in  the case with input saturations and input delays.

This paper is organized as follows. In Section II, we set up the asynchronous sampled-data model and give sufficient conditions for  stability by a Lyapunov-based approach. In Section III, the stability result is applied to solving some asynchronous consensus problems. In Section IV, further extensions of the stability result are discussed. The paper is concluded in Section V.

{\it Notations:} $I_n$ denotes the identity matrix in $\R^{n\times n}$; $\otimes$ denotes the Kronecker  product; $\yi_n=[1,1,\dots,1]^T\in\R^{n}$.

\section{Lyapunov-based asynchronous stability}
In this paper, we are interested in the protocol design and robustness/effectiveness analysis in coordinating agents' states in a general scenario of aperiodic asynchronous sampling and time-varying transmission delays.

Assume that there are $m$ subsystems  in a multi-agent system with zero-order holds and their states are represented by $z_i(t)$, $i=1,2,\dots,m$, respectively, in a common state space $\R^{N}$. The dynamics of subsystem $i$ is given as follows:
\begin{equation}\label{eq:sysi_z}
  \dot{z}_i(t)=Az_i(t)-\sum_{j=1}^m g_{ij}K\hat{z}_j(t),
\end{equation}
where $A\in\R^{N\times N}$ is the common system matrix of subsystems,  $g_{ij}$ represents the coupling weight,   $K\in \R^{N\times N}$ can be viewed as the common state feedback gain, and $\hat{z}_j(t)$ is the sampled state of subsystem $j$.

Let $t^i_k$, $k=0,1,2,\dots$, be the sampling instants of $z_i(t)$,  $\tau^i_k$,  $k=0,1,2,\dots$, be the corresponding time delays  in sampling, transmission, or computation, and  $e^i_k$,  $k=0,1,2,\dots$, be the measurement errors. The sampled state $\hat{z}_i(t)$, incorporating time delays, is given by \begin{equation*}
  \hat{z}_i(t)=z_i(t^i_k)-e^i_k,  t\in[t^i_k+\tau^{i}_k,t^i_{k+1}+\tau^{i}_{k+1}),~k=0,1,2,\dots,
\end{equation*}
which is a piece-wise constant function of time  $t$. Here, the sequences of $\{t^i_k\}$, $\{\tau^i_k\}$, and $\{e^i_k\}$\footnote{Notations $\{t^i_k\}$, $\{\tau^i_k\}$, and $\{e^i_k\}$ stand for the sequences $t^i_0$, $t^i_1$, $\dots$, $\tau^i_0$, $\tau^i_1$, $\dots$, and $e^i_0$, $e^i_1$, $\dots$, respectively.} are completely independent with respect to different $i$, and satisfy the following assumption:
 \begin{assumption} \label{ass:a1} For $k=0,1,2,\dots$,
 \begin{enumerate}
   \item[(1)]$t^{i}_{k+1}-t^{i}_k\leq h$;
   \item[(2)]$\tau^{i}_k<t^i_{k+1}-t^i_k$;
   \item[(3)] $\tau^{i}_k\leq \tau$;
   \item[(4)]  ${e^i_k}^Te^i_k\leq \omega {\hat{z}_i(t^i_k+\tau^i_k)}^T\hat{z}_i(t^i_k+\tau^i_k)$,
 \end{enumerate}
where $h$ is the maximum sampling period, $\tau$ is the maximum time delay, and $\omega\geq 0$ .
\end{assumption}

\begin{remark} \upshape
\begin{enumerate}
\item[(1)] The simultaneous stability model, with each individual represented by system \eqref{eq:sysi_z}, originates from the consensus coordination of multi-agent systems. Several examples will be given in Section III.   Furthermore, this model is also of interest in the networked situation where some individual systems cannot stabilize themselves by their sampled self-states especially when they cannot access their own states due to device failure.  Particularly, when $m=1$, system \eqref{eq:sysi_z} can be viewed a networked control system with aperiodic discrete-time signals and time-varying delays, which is also of its own significance \cite{Zhang_ics2001}.

\item[(2)] Assumption 1~(4) says that the measurement  error is of multiplicative type and proportional to the measurement. The additive errors can be modeled as uncertainties, which will be considered in Section III.
    \item[(3)] It can be shown that Assumption \ref{ass:a1}~(4) can be ensured by
    \begin{equation*}
      {e^i_k}^Te^i_k\leq \frac{\omega}{(1+\sqrt{\omega})^2} z_i(t^i_k)^Tz_i(t^i_k).
    \end{equation*}
\item[(4)]Note that the sequences of $\{t^i_k\}$, $\{\tau^i_k\}$, and $\{e^i_k\}$ are only indexed by the $i$th subsystem. So it can be understood that the information is transmitted via broadcasting. Another interpretation will be given in Section III.
\end{enumerate}
\end{remark}

 Denote $z(t)=[z_1(t)^T~ z_2(t)^T~\dots~z_m(t)^T]^T$ and $\hat{z}(t)=[\hat{z}_1(t)^T~\hat{z}_2(t)^T~\dots~\hat{z}_m(t)^T]^T$. The overall system is given by
\begin{equation}\label{eq:sys_z}
  \dot{z}(t)=(I_m\otimes A)z(t)-(G\otimes K)\hat{z}(t),
\end{equation}
where  $G=[g_{ij}]\in\R^{m\times m}$.

 Let $\lambda_{A_s}$ denote the largest eigenvalue of $(1/2)(A+A^T)$ and let $\sigma_{A}$, $\sigma_{G}$, and $\sigma_{K}$ denote  the largest singular values of $A$, $G$, and $K$, respectively. We have the following result:

\begin{theorem}\label{dl:main}
In system \eqref{eq:sysi_z}, assume that there exists a lower bounded function $V(t)$ with the following property:
\begin{equation}\label{eq:dVdtmain}
  \frac{dV(t)}{dt}\leq -\mu \hat{z}(t)^T\hat{z}(t)+\varepsilon (z(t)-\hat{z}(t))^T\!(z(t)-\hat{z}(t)),
\end{equation}
where $\mu>0$ and $\varepsilon>0$.
If Assumption \ref{ass:a1} holds for all $i$ and there exist some positive numbers $\alpha$ and $\beta$, such that
\begin{equation}\label{eq:dlmaincdtn}
\begin{split}
  \mu &- \varepsilon\omega(1+\frac{1}{\alpha})(1+\frac{1}{\beta}) e^{2\lambda_{A_s}(h+\tau)}\\
  &-\varepsilon\Big((1+\alpha)(1+\frac{1}{\beta}){\sigma_{A}}^2\\
  &+(1+\beta)\frac{7}{3}{\sigma_G}^2{\sigma_K}^2\Big)(h+\tau)^2e^{2\lambda_{A_s}(h+\tau)}>0,
\end{split}
\end{equation}
then
\begin{equation*}
 \lim_{t\to\infty}z(t)=0.
\end{equation*}
\end{theorem}
\begin{proof}
The proof is based on the method of ``analytic synchronization" \cite{Lin_cdc2004,Cao_tac2008} and given in the Appendix.
\end{proof}
\begin{remark}
\begin{enumerate}
  \item[(1)] In \eqref{eq:dVdtmain}, if $\varepsilon\leq 0$, the stability analysis becomes trivial. So we only consider the case with positive $\varepsilon$.
  \item[(2)]  The left side of \eqref{eq:dlmaincdtn} is a continuous function with respect to variables  $h$ and $\tau$. If $h=\tau=0$, then the left side converges to $\mu-\varepsilon\omega$ as parameters $\alpha$ and $\beta$ approach to $\infty$. Therefore, when $\varepsilon\omega<\mu$, from \eqref{eq:dlmaincdtn}, it can be seen that we can always find maximum allowable sampling period $h$ and the corresponding maximum time delay $\tau$ to ensure the stability of the system. However, if $\varepsilon\omega\geq \mu$, then the decrease of $V(t)$ cannot be decided by \eqref{eq:dVdtmain} and thus the stability of system \eqref{eq:sys_z} cannot be decided either.
  \item[(3)]  By Theorem \ref{dl:main},  the maximum allowable $h$ and  $\tau$ can be calculated by \eqref{eq:dlmaincdtn}, where there exists a trade-off between  $\tau$ and $h$. The larger the $h$ is, the smaller the upper bound of allowable time delays is.
\end{enumerate}

\end{remark}

\section{Asynchronous consensus}
Consider a group of $n$ linear autonomous agents interacting with each other through local information flow.  Label these agents with $1$ to $n$ and suppose that the $i$-th agent takes the following dynamics:
\begin{equation}\label{sys:agenti}
  \dot{x}_i(t)=Ax_i(t)+Bu_i(t),
\end{equation}
where $x_i(t)\in\R^{N}$ is the state, $u_i(t)\in \R^{M}$ is an input signal, called {\it protocol} in multi-agent coordination and designed based on the information received from locally linked agents, and  system matrix $A$ and input matrix $B$ are with compatible dimensions.

\subsection{Asynchronous sampling of relative states}
Assume that the interaction topology is modeled by an undirected simple graph $\G$ with vertices $v_1,v_2,\dots,v_n$, which represent the $n$ agents respectively. Assume that there are total $m$ edges, labeled with $1$, $2$, $\cdots$, $m$. The existence of an edge $(v_i,v_j)$ represents that agents $i$ and $j$ have the capacity of knowing the relative state $x_i-x_j$ at the same time, and in such a case, suppose that  agents $i$ and $j$ sample the relative state $x_i-x_j$  at discrete times $t^{ij}_k$,  and they receive the data $x_i(t^{ij}_k)-x_j(t^{ij}_k)$ with  measurement error $e^{ij}_k$ at time $t^{ij}_k+\tau^{ij}_k$, $k=0,1,2,\dots$. Here, $t^{ij}_k=t^{ji}_k$, $\tau^{ij}_k=\tau^{ji}_k$ and $e^{ij}_k=-e^{ji}_k$. For different pair of adjacent agents, the initial sampling of relative states may start at different times, and their subsequent sampling instants are also independent. This is referred to as {\it asynchrony}, which is an intrinsic characteristic in  distributed networks without any global clocks to synchronize agents' actions. Our assumption about time delay $\tau^{ij}_k$ is also quite general.  Time delay $\tau^{ij}_k$ can be changing with respect to $k$ and independent of those on other channels. Let $\N_i$ denote the set of all $j$s, such that $(v_i,v_j)$ exists, and let $t^{ij}(t)=t^{ij}_k$, $e^{ij}(t)=e^{ij}_k$, $t\in[t^{ij}_k+\tau^{ij}_k, t^{ij}_{k+1}+\tau^{ij}_{k+1})$, $k=0,1,2,\dots$. In what follows, if $(v_i,v_j)$ is indexed by $p$ (the $p$-th edge), $t^p_k$, $\tau^p_k$,  and $t^p(t)$ are also used instead of  $t^{ij}_k$, $\tau^{ij}_k$,  and $t^{ij}(t)$.

 With the utilization of zero-order holds, the protocol $u_i(t)$ takes the following form:
\begin{equation}\label{eq:procedge}
  u_i(t)=K\!\sum_{j\in\N_i}(x_j(t^{ij}(t))-x_i(t^{ij}(t))+e^{ij}(t)).
\end{equation}
We will design the feedback matrix $K$ and give sufficient conditions in terms of sampling period and time delays to ensure that $\lim_{t\to\infty} (x_i(t)-(1/n)e^{At}\sum_{j=1}^nx_j(0))=0$, which implies that the system solves the average consensus problem\cite{saber_tac2004}.

 \begin{remark} In the protocol, we require that $t^{ij}_k=t^{ji}_k$, $\tau^{ij}_k=\tau^{ji}_k$ and $e^{ij}_k=-e^{ji}_k$, which ensure the symmetry of information sharing and are the least requirement for distributed average consensus. To get this symmetric information, each pair of adjacent agents should communicate beforehand to synchronize each mutual sampling and ensure the sampled data, like $x_i(t^{ij}_k)-x_j(t^{ij}_k)$, used in both controllers, is updated at the same time.  The synchronization of data sampling and controller update for any pair of adjacent agents is technically possible because it is only based on local communication.  In many engineering applications, such as formation control and attitude alignment, relative states (position, attitude, etc), as a whole, are easily obtainable, which is one of the reasons why we assume the same time delays on both parts of each relative state\cite{Cortes_auto2009}.
\end{remark}

\subsubsection{Edge dynamics\cite{Zelazo_cdc2007edge}}
Assign each edge an arbitrary direction in interaction topology $\G$ and define the $n\times m$ incidence matrix $D=[d_{ij}]$ by (see \cite{Godsil_book2001})
\begin{equation*}
  d_{ij}=\left\{
           \begin{array}{ll}
             1, & \text{if $v_i$ is the head of the oriented edge $j$,}\\
             -1, &  \text{if $v_i$ is the tail of the oriented edge $j$,}\\
             0, & \text{otherwise.}\\
           \end{array}
         \right.
\end{equation*}
Denote $x(t)=[x_1(t)^T~x_2(t)^T~\dots ~x_n(t)^T]^T$ and $z(t)=(D^T\otimes I_N)x(t)$, which is the vector obtained by stacking the relative states corresponding to edges $1$, $2$, $\dots$, $m$ in sequence. Let $z(t)=[z_1(t)^T~z_2(t)^T~\dots~z_m(t)^T]^T$ with $z_p(t)\in\R^N$, $p=1,2,\dots,m$. For the $p$th edge $(v_j,v_i)$, if it is oriented from $v_j$ to $v_i$ in the definition of $D$, denote $z_p(t)=x_i(t)-x_j(t)$, $\hat{z}_p(t)=z_p(t^{p}(t))-e^{ij}(t)$, and  $e^p_k=e^{ij}_k$, $k=0,1,\dots$.  Denote $\hat{z}(t)=[\hat{z}_1(t)^T~\hat{z}_2(t)^T~\dots~\hat{z}_m(t)^T]^T$. The system \eqref{sys:agenti} with protocol \eqref{eq:procedge} has the following equivalent representation:
\begin{equation*}
 \dot{x}(t)=(I_n\otimes A) x(t)-(D\otimes BK )\hat{z}(t),
\end{equation*}
and \begin{equation}\label{eq:dotz}
  \dot{z}(t)=(I_m\otimes A) z(t)-(D^TD\otimes BK )\hat{z}(t).
\end{equation}

 Matrix $DD^T$ is called the {\it graph Laplacian} of $\G$, which is independent of the selection  of $D$ \cite{Horn_book1985}. Let the eigenvalues of $DD^T$ be $\lambda_1$, $\lambda_2$, $\cdots$, $\lambda_n$ in the increasing order. Then  $\lambda_1=0$.   $\lambda_2$ is called the {\em algebraic connectivity} of $\G$, which is positive when $\G$ is connected \cite{Godsil_book2001}. It can be observed that matrices $DD^T$ and $D^TD$ share the same non-zero eigenvalues with the same algebraic multiplicities. So $D^TD$ also serves the purpose of  algebraic characterization of $\G$  and is called {\it edge Laplacian} \cite{Zelazo_cdc2007edge}.

\subsubsection{Average consensus}
The feedback matrix $K$ is designed with the requirement that there exist a positive definite matrix $P$ and a positive number $\mu$ satisfying the following Lyapunov inequalities:
\begin{equation}\label{eq:AlambdaBK}
  (A+\lambda_iBK)^TP +P(A+\lambda_iBK)+2\mu I_N\leq 0,  i=2,3,\dots,n.
\end{equation}
Let $\lambda_{PBK_s}$ denote the largest eigenvalue of $(1/2)(PBK+K^TB^TP)$, $\sigma$ denote the largest singular value of matrix $(D^TD\otimes PBK)-2\mu I_{nN}$,  $\sigma_{BK}$ denote the largest singular value of matrix $BK$, and \begin{equation*}
  \begin{split}
    \mathcal{S}=\{\gamma:&\gamma>0,\mu-{\sigma}/{(2\gamma)}>0,\\
    & {\gamma\sigma}/{2}-\mu+\lambda_n\lambda_{PBK_s} >0\}.
  \end{split}
\end{equation*}
\begin{theorem}\label{dl:relative}
In system \eqref{sys:agenti}, assume that $(A,B)$ is stabilizable, the interaction topology $\G$ is connected, and Assumption~\ref{ass:a1} holds for any $i$.
If there exist positive numbers $\alpha$ and  $\beta$  such that
 \begin{equation}\label{eq:dledgecndn}
 \begin{split}
\omega(1+&\frac{1}{\alpha})(1+\frac{1}{\beta}) e^{2\lambda_{A_s}(h+\tau)}\\
  &+\Big((1+\alpha)(1+\frac{1}{\beta}){\sigma_{A}}^2\\
  &+(1+\beta)\frac{7}{3}{\lambda_n}^2{\sigma_{BK}}^2\Big)(h+\tau)^2 e^{2\lambda_{A_s}(h+\tau)}\\
<& \sup_{\gamma\in\mathcal{S}}\frac{\mu-\frac{\sigma}{2\gamma} }{\frac{\gamma\sigma}{2}-\mu+\lambda_n\lambda_{PBK_s}},
\end{split}
 \end{equation}
  then protocol \eqref{eq:procedge} solves the average consensus problem; mathematically,
  \begin{equation*}
  \lim_{t\to \infty}\bigg(x_i(t)- \frac{1}{n}e^{At}\sum_{i=1}^nx_i(0)\bigg)=0.
    \end{equation*}
\end{theorem}

\begin{remark}
\begin{enumerate}
  \item[(1)] In the proof of Theorem 2, we will see that if there exists $\gamma>0$ such that $\mu-{\sigma}/{(2\gamma)}>0$ and ${\gamma\sigma}/{2}-\mu+\lambda_n\lambda_{PBK_s}\leq 0$, then the derivative of the employed Lyapunov candidate is always negative (in the case with nonzero $\hat{z}(t)$; see \eqref{eq:dVdtleq}), and thus the system solves the average consensus problem for any sampling periods, which is usually not possible for the systems with zero-order holds. So Theorem \ref{dl:relative} only considers the case that $\gamma\in\mathcal{S}$.
  \item[(2)]Since $(A,B)$ is stabilizable, we can always find proper matrices $K$, $P$, and positive number $\mu$ such that inequalities \eqref{eq:AlambdaBK} holds. One feasible solution $K=B^TP$ can be  obtained by solving the following Riccati equation in the LQR problem with $0<\lambda\leq \lambda_2$:
\begin{equation}\label{eq:Riccati}
  PA+A^TP-2\lambda PBB^TP=-2\mu I_N.
\end{equation}
Note that the above Riccati equation has been used in the design of consensus protocols in the literature \cite{Qin_tcns2015,Hu_ta2016}. Furthermore, $\lambda$ can be estimated only by the number of agents without the knowledge of the algebraic connectivity of $\G$\cite{Mohar_gca1991}.
\end{enumerate}
\end{remark}

{\it Proof of Theorem \ref{dl:relative}:}
Consider the following Lyapunov candidate:
\begin{equation*}
  V(t)=\frac{1}{2}z(t)^T(I_m\otimes P)z(t).
\end{equation*}
Then by \eqref{eq:dotz},
\begin{align}
\frac{dV(t)}{dt}=&z(t)^T(I_m\otimes P)\dot{z}(t)\nonumber\\
  =&z(t)^T\Big(I_m\otimes \frac{1}{2}(PA+A^TP)-D^TD\otimes PBK\Big)z(t)\nonumber\\
  &+z(t)^T(D^TD\otimes PBK)(z(t)-\hat{z}(t))\label{eq:dVz}.
\end{align}
Let $D D^T=C^{-1}\Lambda C$, where $C^{-1}=C^T$ and $\Lambda$ is the diagonal matrix with eigenvalues of $D^TD$ in the diagonal positions.  In  equation \eqref{eq:dVz}, we have that
\begin{align}
 z(t)^T\Big(&I_m\otimes\frac{1}{2}(PA+A^TP)-D^TD\otimes PBK\Big)z(t)\nonumber\\
 =& x(t)^T(C^{-1}\otimes I_N)(\Lambda\otimes \frac{1}{2}(PA+A^TP)\nonumber\\
 &+\Lambda^2\otimes PBK)(C\otimes I_N)x(t)\nonumber\\
\leq&-x(t)^T(C^{-1}\otimes I_N)(\Lambda \otimes  \mu I_N)(C\otimes I_N)x(t)\nonumber\\
 \leq&-\mu z(t)^Tz(t). \label{eq:zlambda2}
\end{align}
Substituting  inequality \eqref{eq:zlambda2} into equation  \eqref{eq:dVz} gives that
\begin{align*}
  \frac{dV(t)}{dt}\leq& -\mu z(t)^Tz(t) +z(t)^T(D^TD\otimes PBK)(z(t)-\hat{z}(t))\label{eq:dvdtzedge}\\
  = & -\mu \hat{z}(t)^T\hat{z}(t)\nonumber\\
  &+\hat{z}(t)^T((D^TD\otimes PBK)-2\mu I_{nN})(z(t)-\hat{z}(t))\nonumber\\
  &+(z(t)\!-\!\hat{z}(t))^T\!(D^T\!D\!\otimes\! PBK\!-\!\mu I_{nN})(z(t)\!-\!\hat{z}(t)).\nonumber
\end{align*}
Thus, for any $\gamma>0$,
\begin{equation}\label{eq:dVdtleq}
\begin{split}
  \frac{dV(t)}{dt}\leq&-(\mu-\frac{\sigma}{2\gamma}) \hat{z}(t)^T\hat{z}(t)\\ &+\!(\frac{\gamma\sigma}{2}\!-\!\mu\!+\!\lambda_n\lambda_{PBK_s})(z(t)\!-\!\hat{z}(t))^T(z(t)\!-\!\hat{z}(t)).
\end{split}
\end{equation}

For any $\gamma\in\mathcal{S}$, by Theorem \ref{dl:main}  and inequality  \eqref{eq:dledgecndn}, $\lim_{t\to\infty}z(t)=0$, and thus $\lim_{t\to\infty}V(t)=0$. Let
\begin{equation}\label{eq:delta}
  \delta(t)=x(t)-\yi_n\otimes \frac{1}{n}e^{At}\sum_{i=1}^nx_i(0).
\end{equation}
 Noticing that $z(t)=(D^T\otimes I_N)\delta(t)$ and $\delta(t)\bot \yi_{nN}$, we obtain that
\begin{equation*}
  V(t)=\frac{1}{2}\delta(t)^TDD^T\delta(t)\geq \frac{\lambda_2}{2}\delta(t)^T\delta(t).
\end{equation*}
Therefore, $\lim_{t\to\infty}\delta(t)=0$.
\endproof

\subsection{Asynchronous broadcasting of self-sampled data}
Consider the multi-agent system \eqref{sys:agenti} with an undirected interaction topology $\G$. Different from the  sampling of relative states in the previous subsection, we assume that each agent samples its own state with measurement errors and then broadcasts them to its neighbors with time delays. When all neighbors get the information, the agent and the neighbors all update their controllers. Denote the sampling instants of agent $i$, the associated measurement errors and time delays by $t^i_k$, $e^i_k$, and $\tau^i_k$, $k=0,1,2,\dots$, respectively. Denote $\hat{x}_i(t)=x_i(t^i_k)-e^i_k$ for $t\in[t^i_k+\tau^i_k, t^i_{k+1}+\tau^i_{k+1})$, $k=0,1,2,\dots$. Under Assumption~\ref{ass:a1}~(1-3), we give the following protocol:
\begin{equation}\label{eq:procagent}
  u_i(t)=K\sum_{j\in\N_i}(\hat{x}_j(t)-\hat{x}_i(t)).
\end{equation}
The matrix $K$ will be  designed later.

Denote $x(t)=[x_1(t)^T~x_2(t)^T~\dots~x_n(t)^T]^T$  and $\hat{x}(t)=[\hat{x}_1(t)^T~ \hat{x}_2(t)^T~\dots~\hat{x}_n(t)^T]^T$.
Substituting equation \eqref{eq:procagent} into equation \eqref{sys:agenti} gives that
\begin{equation}\label{eq:dxbrdcst}
  \dot{x}(t)=(I_n\otimes A) x(t)-(DD^T\otimes BK)\hat{x}(t).
\end{equation}

Let $\delta(t)=[\delta_1(t)^T~\delta_2(t)^T~\dots~\delta_n(t)^T]^T$ be  defined by \eqref{eq:delta} with $\delta_i(t)\in\R^N$, and let $\kappa(t)=(1/n)e^{At}\sum_{i=1}^n x_i(0)$. For $t\in[t^i_k+\tau^i_k,t^i_{k+1}+\tau^i_{k+1})$, let $\hat{\delta}_i(t)=\delta_i(t^i_k)-e^i_k$ and $\hat{\kappa}_i(t)=\kappa(t^i_k)$, $k=0,1,\dots$, $i=1,2,\dots,n$. Denote $\hat{\delta}(t)=[\hat{\delta}_1(t)^T~\hat{\delta}_2(t)^T~\dots~\hat{\delta}_n(t)^T]^T$ and $\hat{\kappa}(t)=[\hat{\kappa}_1(t)^T~\hat{\kappa}_2(t)^T~\dots~\hat{\kappa}_n(t)^T]^T$.
Then we have
\begin{equation}\label{eq:dotdeltabrdcst}
  \dot{\delta}(t)=(I_n\otimes A)\delta(t)-(DD^T\otimes BK)(\hat{\delta}(t)+\hat{\kappa}(t)),
\end{equation}
and $\delta(t)\perp \yi_{nN}$.

\begin{remark}\begin{enumerate}
                \item[(1)] By the definition of incidence matrix $D$, if the system solves a consensus problem, then the final consensus trajectory should be $\kappa(t)$. However, it does not satisfy the differential equation \eqref{eq:dxbrdcst} with asynchronous sampling and unstable system matrix $A$; that is, $(DD^T\otimes BK)\hat{\kappa}(t)\not=0$ in most cases. Thus in such cases, the asynchronous consensus cannot be achieved.
                \item[(2)]  In the case of multiplicative type of measurement errors, even if the consensus is reached, the final consensus state is usually not zero and thus $e^i_k$ may never converge to zero as $k\to\infty$. Therefore, the multiplicative type of measurement errors can destroy the consensus of the considered system \eqref{eq:dxbrdcst}.
              \end{enumerate}

\end{remark}

\subsubsection{Asynchronous consensus of single-integrators}
In the case of multiple single-integrators with $N=1$, $A=0$, and $B=1$, $\kappa(t)$ is time-invariant and so is $\hat{\kappa}(t)$. Therefore, choose $K=1$ and by \eqref{eq:dotdeltabrdcst},
\begin{equation*}
  \dot{\delta}(t)=-DD^T\hat{\delta}(t).
\end{equation*}
Consider the following Lyapunov candidate
\begin{equation}\label{eq:Vdelta2}
V(t)=\frac{1}{2}\delta(t)^T\delta(t).
\end{equation}
With the same arguments as in proving \eqref{eq:dVdtleq}, we have
\begin{equation*}
\begin{split}
  \frac{dV(t)}{dt}\leq &- \lambda_2\delta(t)^T\delta(t)+ \delta(t)^TDD^T(\delta(t)-\hat{\delta}(t))\\
  \leq &- (\lambda_2-\frac{\sigma}{2\gamma})\hat{\delta}(t)^T\hat{\delta}(t)\\
  &+(\frac{\gamma\sigma}{2}+\lambda_n-\lambda_2)(\hat{\delta}(t)-\delta(t))^T(\hat{\delta}(t)-\delta(t)),
\end{split}
\end{equation*}
where $\sigma=\max\{2\lambda_2,\lambda_n-2\lambda_2\}$ is the largest singular value of matrix $DD^T-2\lambda_2I_n$, and $\gamma>0$ with $\lambda_2-{\sigma}/(2\gamma)>0$.

By Theorem \ref{dl:main}, we have the following result:
\begin{theorem}\label{dl:singleintegrator:brdcst} In system \eqref{sys:agenti}, suppose that each agent takes the single-integrator dynamics without measurement errors, and suppose that the sampling instants and  time delays satisfy Assumption~\ref{ass:a1}~(1-3), $i=1,2,\dots,n$.
If the interaction topology $\G$ is connected and
\begin{equation}\label{eq:sngintgrtrcndtn}
\begin{split}
(h+\tau)^2<\frac{3}{7{\lambda_n}^2}\sup_{\gamma>0}\frac{\lambda_2-\frac{\sigma}{2\gamma}}{\frac{\gamma\sigma}{2}+\lambda_n-\lambda_2},
\end{split}
\end{equation}
then protocol \eqref{eq:procagent} with $K=1$ solves the average consensus problem.
\end{theorem}
\begin{remark}\label{rmk:periodic}
In \cite{Xie_acc2009}, it has been shown that the necessary and sufficient condition for the synchronous sampled-data consensus of single-integrators with sampling period $h$ is $0<h\leq {2}/{\lambda_n}$. Compared with this condition, the condition \eqref{eq:sngintgrtrcndtn} may be a bit conservative, which is understandable since this condition is valid in the general asynchronous sampling case and obtained based on the general model of linear subsystems.
\end{remark}
\subsubsection{Asynchronous consensus of marginally stable systems}

Suppose that matrix $A$ is marginally stable.  Then $ \max\nolimits_s\|e^{As}\|_2$ is bounded and by Lemma \ref{yl:singularvalue}, we have\footnote{In case of a $0$-valued $\lambda_{A_s}$ serving as a denominator,  the value of the expression is set to the limit when $\lambda_{A_s}\to 0$. }
\begin{equation*}
\begin{split}
  \|\hat{\kappa}(t)-\yi_n\otimes \kappa(t)\|_2 < &(1/\sqrt{n})\|\sum_{i=1}^n x_i(0)\|_2 \max\nolimits_s\|e^{As}\|_2\\
  &\times \frac{\sigma_A(e^{\lambda_{A_s} h}-1)}{\lambda_{A_s}}.
\end{split}
\end{equation*}
Denote the right side of the above equation by $\Delta_{\kappa}(h)$. Clearly, $\lim_{h\to 0} \Delta_{\kappa}(h)=0$.

For any $i$, denote $e_i(t)=e^i_k$ and $\tilde{\delta}_i(t)=\delta_i(t^i_k)$, $t\in[t^i_k+\tau^i_k,t^i_{k+1}+\tau^i_{k+1})$, $k=0,1,\dots$. Denote $e(t)=[e_1(t)^T~e_2(t)^T$ $\dots$ $e_n(t)^T]^T$ and $\tilde{\delta}(t)=[\tilde{\delta}_1(t)^T~\tilde{\delta}_2(t)^T~\dots~\tilde{\delta}_n(t)^T]^T$. Then $\hat{\delta}(t)=\tilde{\delta}(t)-e(t)$ and
\begin{equation}\label{eq:dotdeltamrgnl}
\begin{split}
  \dot{\delta}(t)=&(I_n\otimes A)\delta(t)-(DD^T\otimes BK)\tilde{\delta}(t)\\
  &+ (DD^T\otimes BK)(e(t)-\hat{\kappa}(t)).
\end{split}
\end{equation}
If $\|e(t)\|_2\leq \Delta_e$, then
\begin{equation*}
\begin{split}
  \|(DD^T&\otimes BK)(e(t)-\hat{\kappa}(t))\|_2\\
  &\leq  \|DD^T\otimes BK\|_2(\Delta_{\kappa}(h)+\Delta_e)
\triangleq \Delta(h).
\end{split}
\end{equation*}

Choose the feedback matrix $K=B^TP$  with a positive definite matrix $P$ and a positive number $\mu$ by \eqref{eq:Riccati} and consider the following Lyapunov function candidate
\begin{equation*}
V(t)=\frac{1}{2}\delta(t)^T(I_n\otimes P)\delta(t).
\end{equation*}
Then we have
\begin{align}
  \frac{dV(t)}{dt}\leq &\delta(t)^T(I_n\otimes PA)\delta(t)\!-\! \delta(t)^T(DD^T\otimes PBB^TP)\tilde{\delta}(t)\nonumber\\
    &+ \frac{\lambda_P}{2\eta} \delta(t)^T \delta(t)+\frac{\lambda_P\eta}{2}\Delta(h)^2\nonumber\\
    \leq&-(\mu-\frac{\lambda_P}{2\eta}-\frac{\sigma}{2\gamma}) \tilde{\delta}(t)^T\tilde{\delta}(t) \nonumber\\ &+\!(\frac{\gamma\sigma}{2}\!-\!\mu\!+\!\frac{\lambda_P}{2\eta}\!+\!\lambda_n{\sigma_{PB}}^2)(\delta(t)\!-\!\tilde{\delta}(t))^T\!(\delta(t)\!-\!\tilde{\delta}(t))\nonumber\\
    &+\frac{\lambda_P\eta}{2}\Delta(h)^2, \label{eq:dVdtMrgnl}
\end{align}
where $\lambda_P$ is the largest eigenvalue of $P$, $\sigma_{PB}$ is the largest singular value of $PB$, $\sigma$ is the largest singular value of $(DD^T\otimes PBB^TP)-2(\mu-\lambda_P/(2\eta)) I_{nN}$, and $\gamma$ and $\eta$ are any positive numbers with $\mu- {\lambda_P}/({2\eta})-{\sigma}/({2\gamma})>0$.

It can be seen that equations \eqref{eq:dotdeltabrdcst} and \eqref{eq:dVdtMrgnl} do not exactly match the model \eqref{eq:sys_z}, \eqref{eq:dVdtmain}. However, when the largest  sampling period $h$ and measurement errors are small, the system represented by \eqref{eq:dotdeltamrgnl} and \eqref{eq:dVdtMrgnl} can be seen as the model \eqref{eq:sys_z}, \eqref{eq:dVdtmain} with uncertainties.
Let $\sigma_{BB^TP}$ denote the largest singular value of $BB^TP$,
\begin{equation*}
\begin{split}
  \bar{\Delta}(h,\tau,\alpha,\gamma,\eta)=&(\frac{\gamma\sigma}{2}-\mu+\frac{\lambda_P}{2\eta}+\lambda_n{\sigma_{PB}}^2)(1+\frac{1}{\alpha})\\
  &\times\max\nolimits_s{\|e^{As}\|_{\infty}}^2n(h+\tau)^2{\Delta(h)}^2\\
  &+\frac{1}{2}\lambda_P\eta{\Delta(h)}^2,
\end{split}
\end{equation*}
 and
\begin{equation*}
 \begin{split}
  \Gamma(h,\tau,\alpha,\beta,\gamma,\eta)=&\mu-\frac{\lambda_P}{2\eta}-\frac{\sigma}{2\gamma}\\
  &-(\frac{\gamma\sigma}{2}-\mu+\frac{\lambda_P}{2\eta}+\lambda_n{\sigma_{PB}}^2)(1+\alpha)\\
  &\times\Big((1+\frac{1}{\beta}){\sigma_{A}}^2+(1+\beta)\frac{7}{3}{\lambda_n}^2{\sigma_{BB^TP}}^2\Big)\\
  &\times(h+\tau)^2e^{2\lambda_{A_s}(h+\tau)},
\end{split}
 \end{equation*}
 where $\|\cdot\|_{\infty}$ denotes the maximum row sum matrix norm.
  By the same arguments as in the proof of Theorem \ref{dl:main}, we can obtain the following result:
\begin{theorem}\label{dl:brdcst:error}
In system \eqref{sys:agenti}, assume that $(A,B)$ is stabilizable, $A$ is marginally stable, Assumption~\ref{ass:a1}~(1-3) holds,   $\|e(t)\|_2 \leq\Delta_e$,  $t^i_{k+1}-t^i_{k}$ is lower bounded by a positive number independent of $k$, $i=1,2,\dots,n$, and the interaction topology is connected.  Under the proposed protocol \eqref{eq:procagent} with $K$ given by \eqref{eq:Riccati}, if $\mathcal{S}\not=\emptyset$, then  for any time $t'$ and any $\theta>1$, there exists some $t$, $t\geq t'$, such that
\begin{equation}\label{eq:consensuserror}
  \tilde{\delta}(t)^T\tilde{\delta}(t)\leq \inf_{[\alpha,\beta,\gamma,\eta]\in\mathcal{S}}\frac{\theta\bar{\Delta}(h,\tau,\alpha,\gamma,\eta)}{\Gamma(h,\tau,\alpha,\beta,\gamma,\eta)},
\end{equation}
where
\begin{equation*}
  \begin{split}
   \mathcal{S}=\Big\{[\alpha,\beta,\gamma,\eta]:&\alpha,\beta,\gamma,\eta>0,\\
   &\mu- {\lambda_P}/({2\eta})-{\sigma}/({2\gamma})>0, \\
   &\frac{\gamma\sigma}{2}-\mu+\frac{\lambda_P}{2\eta}+\lambda_n{\sigma_{PB}}^2\geq 0,\\
      &\Gamma(h,\tau,\alpha,\beta,\gamma,\eta)>0\Big\}.
  \end{split}
\end{equation*}
\end{theorem}
\begin{proof}
The proof is omitted here.
\end{proof}

\section{Extensions}
This section only considers the model \eqref{eq:sys_z}, \eqref{eq:dVdtmain}. So the obtained results are also valid for the consensus models in Section III.
\subsection{Quantized sampling}
This subsection gives a simple example to show the effectiveness of Theorem \ref{dl:main} in the scenario of quantization. Further discussions will continue in the next subsection.

In system \eqref{eq:sys_z}, if apply the following Logarithmic quantizer to each entry of $z_i(t^i_k)$\cite{Aldajani_dsp2008}:
\begin{equation}\label{eq:quantizerlog}
  Q_{\log\varepsilon}(\xi)=\left\{
                     \begin{array}{ll}
                       0, & \text{if $\xi=0$,} \\
                      \text{\upshape sign}(\xi)\epsilon^{\lfloor\log_{\epsilon}|\xi|\rfloor}, & \text{otherwise,} \\
                     \end{array}
                   \right.
\end{equation}
then
\begin{equation}\label{eq:z_Qlog}
\begin{split}
 (z_i(t^i_k)-&Q_{\log\epsilon}(z_i(t^i_k)))^T(z_i(t^i_k)-Q_{\log\epsilon}(z_i(t^i_k)))\\
  &\leq (\epsilon-1)^2 Q_{\log\varepsilon}(z_i(t^i_k))^TQ_{\log\epsilon}(z_i(t^i_k)),
\end{split}
\end{equation}
where $\epsilon>1$ is the quantizing level, $\xi$ is the scalar quantizer input,  $\lfloor\cdot\rfloor$ denotes the largest integer not greater than the considered variable, and $Q_{\log\epsilon}(z_i(t^i_k)))$ is the vector obtained by applying the quantizer \eqref{eq:quantizerlog} to  $z_i(t^i_k)$ entrywise. So $Q_{\log\epsilon}(z_i(t^i_k))$ can be viewed as the sampled value $z_i(t^i_k)$ with measurement error $e^i_k=z_i(t^i_k)-Q_{\log\epsilon}(z_i(t^i_k))$.

Redefine $\hat{z}_i(t)=Q_{\log\epsilon}(z_i(t^i_k)))$, $t\in[t^i_k+\tau^i_k,t^i_{k+1}+\tau^i_{k+1})$, $k=0,1,\dots$, $i=1,2,\dots,m$, and define the same $\hat{z}(t)$ as in Section II. By Theorem \ref{dl:main}, we have the following corollary:
\begin{corollary}
In system \eqref{eq:sys_z} with the utilization  of Logarithmic quantizer \eqref{eq:quantizerlog} in sampling, if Assumption \ref{ass:a1} (1-3), $i=1,2,\dots,m$, and inequality \eqref{eq:dVdtmain} hold, and there exist  some positive numbers $\alpha$ and $\beta$, such that
\begin{equation*}
\begin{split}
  \mu &- \varepsilon(\epsilon-1)^2(1+\frac{1}{\alpha})(1+\frac{1}{\beta}) e^{2\lambda_{A_s}(h+\tau)}\\
  &-\varepsilon\Big((1+\alpha)(1+\frac{1}{\beta}){\sigma_{A}}^2\\
  &+(1+\beta)\frac{7}{3}{\sigma_G}^2{\sigma_K}^2\Big)(h+\tau)^2e^{2\lambda_{A_s}(h+\tau)}>0,
\end{split}
\end{equation*}
then \[
 \lim_{t\to\infty}z(t)=0.\]
\end{corollary}

Note that the same property as \eqref{eq:z_Qlog} of the above Logarithmic quantizer can be preserved by  the output of an event-based sampler with event-triggering condition that $e_i(t)^Te_i(t) > (\epsilon-1)^2 \hat{z}_i(t)^T\hat{z}_i(t)$, where $\hat{z}_i(t)$ is the latest measurement of $z_i(\cdot)$ up to $t$ and $e_i(t)$ is the measurement error. So similar discussions in the following subsection can be given to quantized systems.
\subsection{Event-triggered sampling}
This subsection shows an interesting application of asynchronous aperiodic sampling in removing the Zeno behavior in event-based control.

Consider the multi-agent system with $m$ subsystem in the absence of time delays and introduce the same notations $z_i(t)$ and $z(t)$ as in Section II to represent the state of agents. Let $\hat{z}_i(t)$ be the measurement of $z_i(t)$ and it is updated in an event-triggered way. In detail, for agent $i$, let $\hat{z}_i(t)$ first updated at $t^i_0$ and set $\hat{z}_i(t^i_0)=z_i(t^i_0)$. For $k=0,1,\dots$, as time increases from $t^i_k+h$, agent $i$ examine the condition that
\begin{equation}\label{eq:eventcdtn}
  (z_i(t)-z_i(t^i_k))^T(z_i(t)-z_i(t^i_k))\geq \omega z_i(t^i_k)^Tz_i(t^i_k),
\end{equation}
where $h$ is a predetermined positive dwell time and called {\it rest time} in \cite{Xiao_scl2016}.
For the first time $t$ when the above inequality become correct, set $t^i_{k+1}=t$ and $\hat{z}_i(t^i_{k+1})={z}_i(t^i_{k+1})$. Note that the sequence $t^i_0$, $t^i_1$, $\dots$ may terminate at some finite $k$. Obviously, $\hat{z}_i(t)$ is piece-wise constant; that is, in the case with $t^i_k$ given  through the above procedure,   $\hat{z}_i(t)=z_i(t^i_k)$ for $t\in[t^i_k,t^i_{k+1})$ if $t^i_{k+1}$ exists; and $\hat{z}_i(t)=z_i(t^i_k)$ for $t\in[t^i_k,\infty)$ if $t^i_{k+1}$ doesn't exist.

Denote $\hat{z}(t)=[\hat{z}_i(t)^T~\hat{z}_2(t)^T~\dots ~ \hat{z}_m(t)^T]^T$. The overall system, represented by \eqref{eq:sys_z}, has the following stability result:
\begin{corollary}\label{dl:eventbased}
In system \eqref{eq:sys_z}, assume that  $\hat{z}(t)$ is updated according to the event-triggering condition \eqref{eq:eventcdtn} with dwell time $h$ and without time delays, and  there exists a lower bounded function $V(t)$ with property described by \eqref{eq:dVdtmain}.
If there exist some positive numbers $\alpha$ and $\beta$, such that
\begin{equation*}
\begin{split}
  \mu &- \varepsilon\omega(1+\frac{1}{\alpha})(1+\frac{1}{\beta}) e^{2\lambda_{A_s}h}\\
  &-\varepsilon\Big((1+\alpha)(1+\frac{1}{\beta}){\sigma_{A}}^2\\
  &+(1+\beta)\frac{7}{3}{\sigma_G}^2{\sigma_K}^2\Big)h^2e^{2\lambda_{A_s}h}>0,
\end{split}
\end{equation*}
then
\begin{equation*}
 \lim_{t\to\infty}z(t)=0.
\end{equation*}
\end{corollary}
\begin{proof}
Note that the definition of $t^i_{k}$ is different from that in Section II and Assumption~1~(1) doesn't hold. However, by event-triggering condition \eqref{eq:eventcdtn}, for any $t$, $t\not\in\cup_k(t^i_k,t^i_k+h)$,
\begin{equation}\label{eq:eventerror}
 (z_i(t)-\hat{z}_i(t))^T(z_i(t)-\hat{z}_i(t))< \omega \hat{z}_i(t)^T\hat{z}_i(t).
\end{equation}
So we can view $\hat{z}_i(t)$ as the sampled state of $z_i(t)$ at time $t$ with measurement error $z_i(t)-\hat{z}_i(t)$, which is exactly Assumption 1~(4).

(Case 1) If time sequence $t^i_0$,  $t^i_1$, $\dots$, is of finite length, let $\bar{k}=\text{arg}\max_k{t^i_k}$. For any $k$, $k<\bar{k}$, there exists a positive integer $n^i_k$ such that $t^i_k+n^i_kh<t^i_{k+1}\leq t^i_{k}+(n^i_{k}+1)h$. Let $s^i_k$, $k=0,1,2,\dots$, denote all the time instants $t^i_k+ph$, $k=0,2,\dots$, $\bar{k}-1$, $p=0,1,\dots,n^i_k$, and $t^i_{\bar{k}}+ph$, $p=0,1,\dots$, in the increasing order.  (Case 2)  If time sequence $t^i_0$,  $t^i_1$, $\dots$, is of infinite length, for any $k$, there exists a positive integer $n^i_k$ such that $t^i_k+n^i_kh<t^i_{k+1}\leq t^i_{k}+(n^i_{k}+1)h$. Let $s^i_k$, $k=0,1,2,\dots$, denote all the time instants $t^i_k+ph$, $k=0,2,\dots$,  $p=0,1,\dots,n^i_k$, in the increasing order.

Then we have that $s^i_{k+1}-s^i_{k}\leq h$ and inequality \eqref{eq:eventerror} holds at any $t=s^i_k$. By Theorem \ref{dl:main}, we have the stability of the system.
\end{proof}

\begin{remark}
\begin{enumerate}
  \item[(1)] The idea of removing the Zeno behavior in event-based control by adding a dwell time was previously presented for networks of single-integrators in \cite{Xiao_scl2016}.
  \item[(2)] Event-triggering condition \eqref{eq:eventcdtn} can be replaced by
\begin{equation*}
  \|z_i(t)-z_i(t^i_k)\|_2\geq \frac{\sqrt{\omega}}{1+\sqrt{\omega}} \|z_i(t)\|_2,
\end{equation*}
which is classified as a  quadratic event-triggering condition in \cite{Heemels_tac2013}.
\end{enumerate}

\end{remark}

 \subsection{Saturated control with input delay}
This subsection discusses the extension of Theorem \ref{dl:main} in the case with input saturations and input delays\cite{JiangZP_scl2013,TYang_2014,SuHS_ijrnc2015,Lim_tac2015,LinZL_ijrnc2016,yuedong_ijc2016 }.

In system \eqref{eq:sysi_z}, $\sum_{j=1}^m g_{ij}K\hat{z}_j(t)$ can be viewed as the input. Instead of applying a saturation function to the input or to each sampled state $\hat{z}_j(t)$, we use a scaler function $\rho(\cdot)$, defined below, to scale down the sampled states to meet the restriction of input saturation:
\begin{equation*}
  \rho(\xi)=\left\{
                       \begin{array}{ll}
                         1, & \mbox{if $\xi=0$}, \\
                         \frac{1}{\lceil\frac{\|\xi\|_{\infty}}{\rho_s}\rceil}, & \mbox{otherwise,} \\
                       \end{array}
                     \right.
\end{equation*}
where $\xi\in\R^N$, $\lceil\cdot\rceil$ denotes the ceiling function (which returns the smallest integer greater than or equal the considered variable) and $\rho_s$ is some positive number decided by the maximum allowable magnitude of input. Clearly, for any $\xi$, $0<\rho(\xi)\leq 1$ and $\|\rho(\xi)\xi\|_{\infty}\leq \rho_s$. Function $\rho(\xi)$ can be also defined in terms of other norms. Then the subsystem $i$ takes the following dynamics:
\begin{equation}\label{eq:syssaturationzi}
  \dot{z}_i(t)=Az_i(t)-\sum_{j=1}^m g_{ij}\rho(\hat{z}_j(t-\tau_{\text{in}}))K\hat{z}_j(t-\tau_{\text{in}}),
\end{equation}
where $\tau_{\text{in}}$ is the input delay.
\begin{theorem}\label{dl:saturation}
Assume that Assumption \ref{ass:a1} holds for all $i$ and there exists a lower bounded function $V(t)$ such that
\begin{equation}\label{eq:dVdtsaturation}
\begin{split}
  \frac{dV(t)}{dt}\leq &-\mu \hat{z}(t-\tau_{\text{in}})^T\hat{z}(t-\tau_{\text{in}})\\
  &+\varepsilon (z(t)-\hat{z}(t-\tau_{\text{in}}))^T\!(z(t)-\hat{z}(t-\tau_{\text{in}})),
\end{split}
\end{equation}
where $\mu>0$ and $\varepsilon>0$.
If there exist some positive numbers $\alpha$ and $\beta$, such that
\begin{equation}\label{eq:saturationcdtn}
\begin{split}
  \mu &- \varepsilon\omega(1+\frac{1}{\alpha})(1+\frac{1}{\beta}) e^{2\lambda_{A_s}(h+\tau+\tau_{\text{in}})}\\
  &-\varepsilon\Big((1+\alpha)(1+\frac{1}{\beta}){\sigma_{A}}^2\\
  &+(1+\beta)\frac{7}{3}{\sigma_G}^2{\sigma_K}^2\Big)(h+\tau+\tau_{\text{in}})^2e^{2\lambda_{A_s}(h+\tau+\tau_{\text{in}})}>0,
\end{split}
\end{equation}
then
\begin{equation*}
 \lim_{t\to\infty}z(t)=0.
\end{equation*}
\end{theorem}
\begin{proof}
The proof follows the same procedure as in the proof of Theorem \ref{dl:main}.  Here, $\tau^i_k+\tau_{\text{in}}$ is viewed as the time delay in sampling $z_i(t)$. We also use the property that
the largest singular value of $G$ is not less than that of matrix $[g_{ij}\rho(\hat{z}_j(t-\tau_{\text{in}}))]$.
\end{proof}

Note that neither $\rho(\cdot)$ nor $\rho_s$ appears in Theorem \ref{dl:saturation}. However, the saturated input with function $\rho(\cdot)$ as well as initial states has a significant  impact on the existence of a valid feedback gain $K$ and a function $V(t)$ satisfying inequality \eqref{eq:dVdtsaturation}. For example, when we know that the states of all subsystem \eqref{eq:syssaturationzi} are bounded by Lyapunov methods, then $\rho(\hat{z}_j(t-\tau_{\text{in}}))$ has a finite set of all possible values. The considered system is a switched system with state-dependent switching between a finite number of subsystems. $V(t)$ can be chosen to be a common Lyapunov function\cite{Sun_book2011}. Initial states are closely related to the upper bound of system states and thus affect the number of subsystems in the resulting switched system.

\begin{example}
Consider the network of single-integrators described by equation \eqref{sys:agenti}, where $N=1$, $A=0$, and $B=K=1$. Revise protocol \eqref{eq:procedge} by multiplying each sampled state by the function $\rho(\cdot)$. Denote $\Lambda(t)=\diag([\rho(\hat{z}_1(t-\tau_{\text{in}}))~\rho(\hat{z}_2(t-\tau_{\text{in}}))~\dots~\rho(\hat{z}_m(t-\tau_{\text{in}}))])$, and suppose that $\rho_m=\min\{\rho(\hat{z}_i(t-\tau_{\text{in}})):i=1,2,\dots,m\}$ is larger than $0$. We obtain that
 \begin{equation*}
   \dot{x}(t)=-D\Lambda(t)\hat{z}(t-\tau_{\text{in}}).
 \end{equation*}
 Consider the Lyapunov candidate given by \eqref{eq:Vdelta2}.
We have
\begin{equation*}
  \begin{split}
     \frac{dV(t)}{dt} =& -z(t)\Lambda(t)\hat{z}(t-\tau_{\text{in}})\\
       = & -\hat{z}(t-\tau_{\text{in}})^T\Lambda(t)\hat{z}(t-\tau_{\text{in}})\\
       &-\hat{z}(t-\tau_{\text{in}})^T\Lambda(t)(z(t)-\hat{z}(t-\tau_{\text{in}}))\\
       \leq & -(\rho_m-\frac{1}{2\gamma})\hat{z}(t-\tau_{\text{in}})^T\Lambda(t)\hat{z}(t-\tau_{\text{in}})\\
       &+\frac{\gamma}{2}(z(t)-\hat{z}(t-\tau_{\text{in}}))^T(z(t)-\hat{z}(t-\tau_{\text{in}})),
   \end{split}
\end{equation*}
where parameter $\gamma>{1}/{(2\rho_m)}$.

\end{example}

\section{Numerical Examples}
\begin{figure}[h]\centering
      \includegraphics[width=5cm]{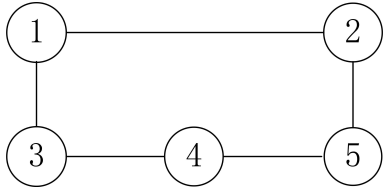}
      \caption{Interaction Topology}
      \label{fig:topology}
\end{figure}
This section presents three numerical examples to demonstrate the effectiveness of the theoretical results.

\begin{example}[Quantized sampling]\label{lz:1}
We first consider the network of harmonic oscillators under the topology depicted in Fig.~\ref{fig:topology}. These oscillators take the dynamics given by \eqref{sys:agenti}. The system matrix $A$ and input $B$ are given by
\begin{equation*}
  A=\left[
      \begin{array}{cc}
        0 & 1 \\
        -1 & 0 \\
      \end{array}
    \right],\quad
B=\left[
      \begin{array}{c}
        0\\
        1\\
      \end{array}
    \right].
\end{equation*}
Let $\lambda=\lambda_2$ and $\mu=1$. Solving equation \eqref{eq:Riccati}, we obtain $K=[0.5626~1.0633]$.
Apply the Logorithmic quantizer \eqref{eq:quantizerlog} with $\epsilon=1.1$ to the relative states in protocol \eqref{eq:procedge}. By Theorem \ref{dl:relative}, we obtain that if $h+\tau<0.017$, then the system solves the average consensus problem. Let the sampling periods randomly generated in  $[0.005,0.012]$ and the delays are random and not larger than $0.005$. The states of oscillators and quantized relative states with time delays are given in Fig.~\ref{fig:1_1} and Fig.~\ref{fig:1_2}
\begin{figure}[h]\centering
      \includegraphics[width=6cm]{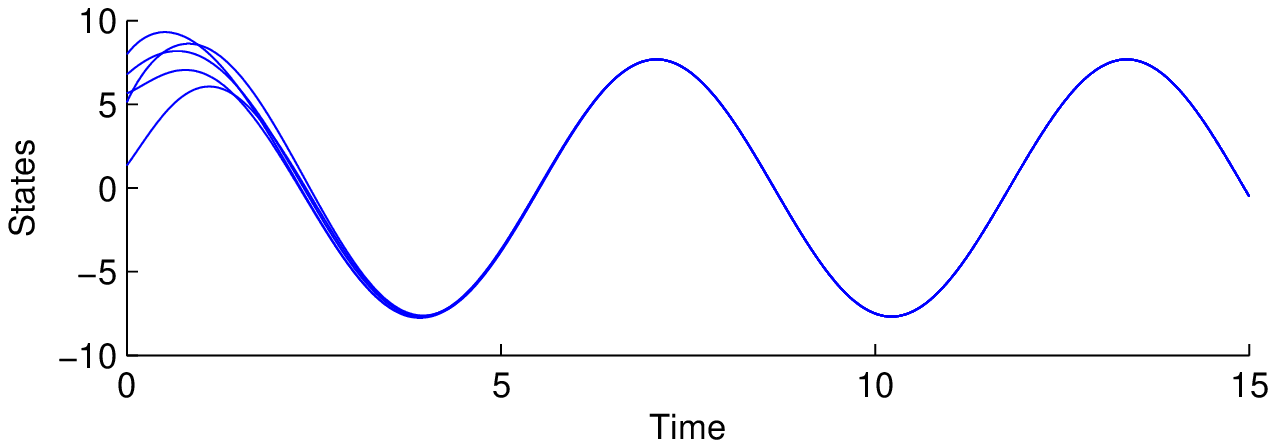}\\
      \includegraphics[width=6cm]{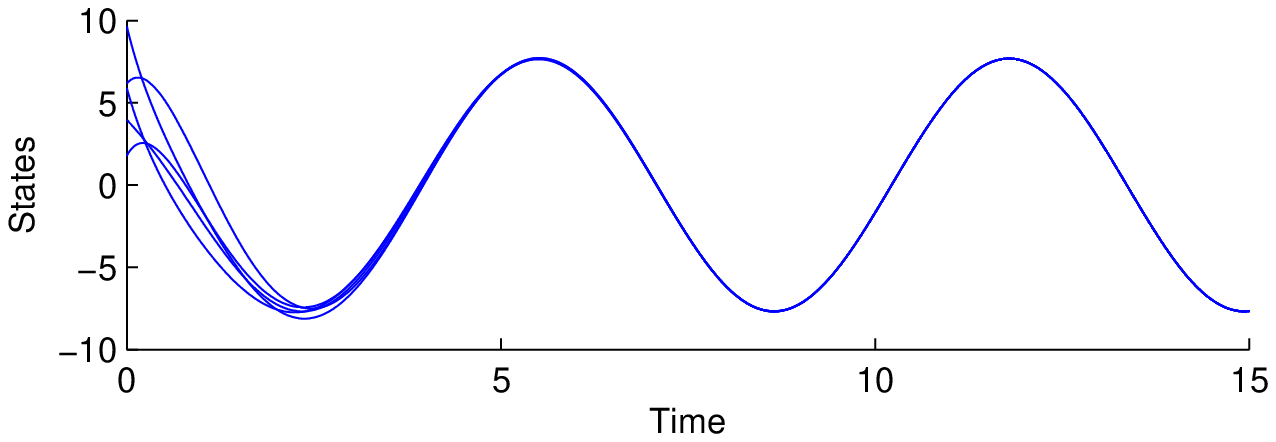}
      \caption{State trajectories of oscillators. The two figures show the first and second components of states, respectively.}
      \label{fig:1_1}
\end{figure}
\begin{figure}[h]\centering
      \includegraphics[width=6cm]{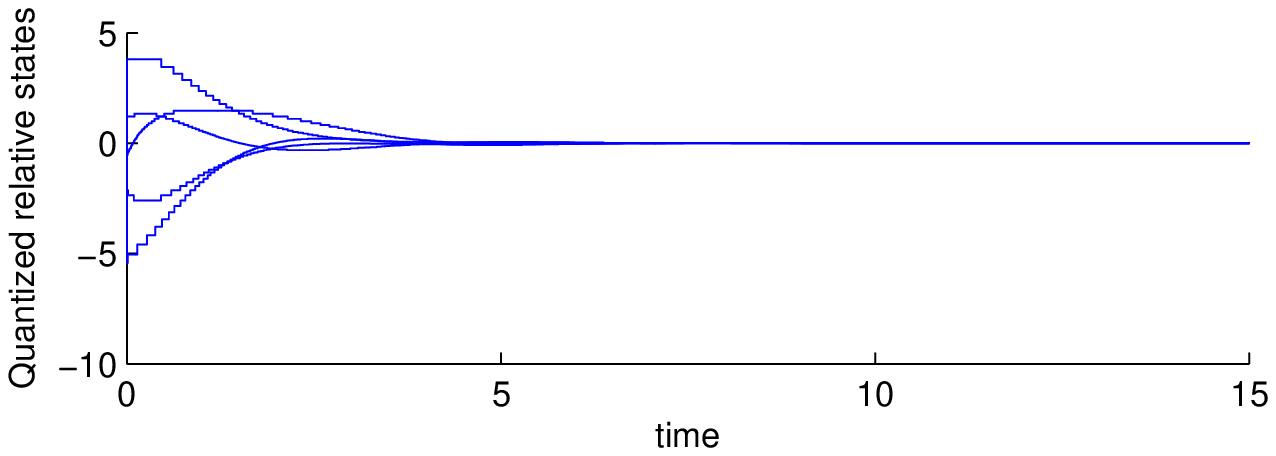}\\
      \includegraphics[width=6cm]{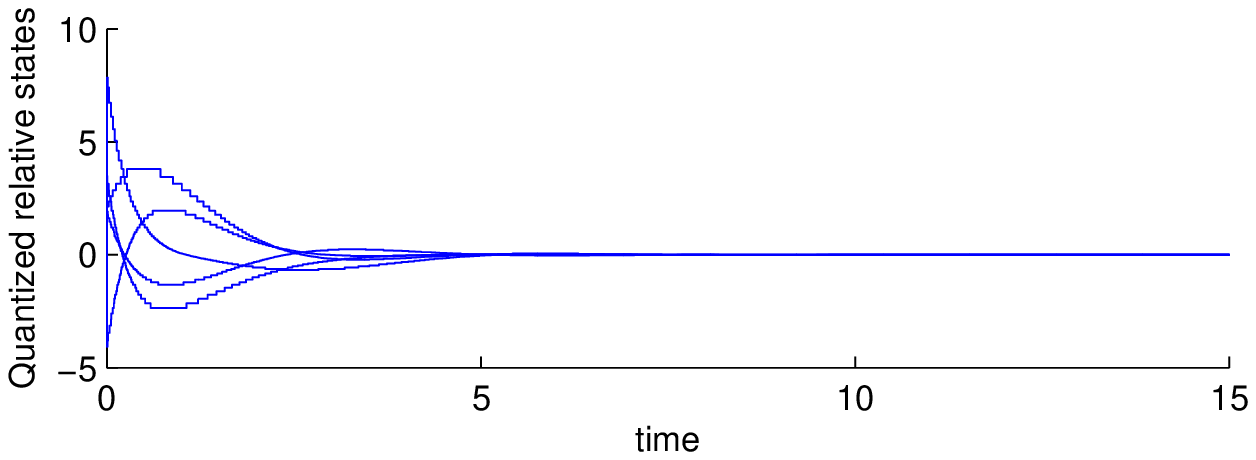}
      \caption{Quantized relative states with time delays. The two figures show the first and second components of quantized relative states, respectively.}
      \label{fig:1_2}
\end{figure}
\end{example}

\begin{example}\label{lz:2}
Consider the network of single integrators under the topology depicted in Fig.~\ref{fig:topology}. If protocol \eqref{eq:procagent} is used, then by Theorem \ref{dl:singleintegrator:brdcst}, $h+\tau<0.0691$ is sufficient for state consensus in the absence of measurement errors. By the result in \cite{Xie_acc2009} (see Remark \ref{rmk:periodic}), $h\leq 0.5528$ is a necessary and sufficient condition for state consensus in the synchronous periodic sampling case without time delays and measurement errors. This result is  not applicable to our system because of different model setups and assumptions. Moreover, no necessary and sufficient results have been derived for the asynchronous time-delayed system studied in this paper in the literature.
\end{example}

\begin{example}[Event-triggered broadcasting]\label{lz:3}
Consider the system in Example \ref{lz:2} with  $h=0.025$ and $\tau=0.02$. Denote $\hat{x}_i((t^i_k+\tau^i_k)^-)=\lim_{t\uparrow t^i_k+\tau^i_k}\hat{x}_i(t)$, which denotes the most recent broadcasted data of $x_i(t)$ before time $t^i_k$. If the following event-triggering condition holds:
\begin{equation*}
|x_i(t^i_k)-\hat{x}_i((t^i_k+\tau^i_k)^-)|> \min\{0.3 |\hat{x}_i((t^i_k+\tau^i_k)^-)|,0.08\},
\end{equation*}
then $x_i(t^i_k)$ is broadcasted to the neighbors of agent $i$. Then we have that  $\Delta(h)=0.2894$.
By Theorem \ref{dl:brdcst:error}, the system solves the average consensus problem with the error $0.4535$ given by \eqref{eq:consensuserror}, where $\alpha=0.5$, $\gamma=3.188$, and $\eta=1.6$.  Let the sampling periods randomly generated in $[0.02,0.025]$ and the delays are random and
not larger than 0.02. The state trajectories and event numbers are given in Fig.~\ref{fig:3_1} and Fig.~\ref{fig:3_2}, respectively.
\begin{figure}[h]\centering
      \includegraphics[width=6cm]{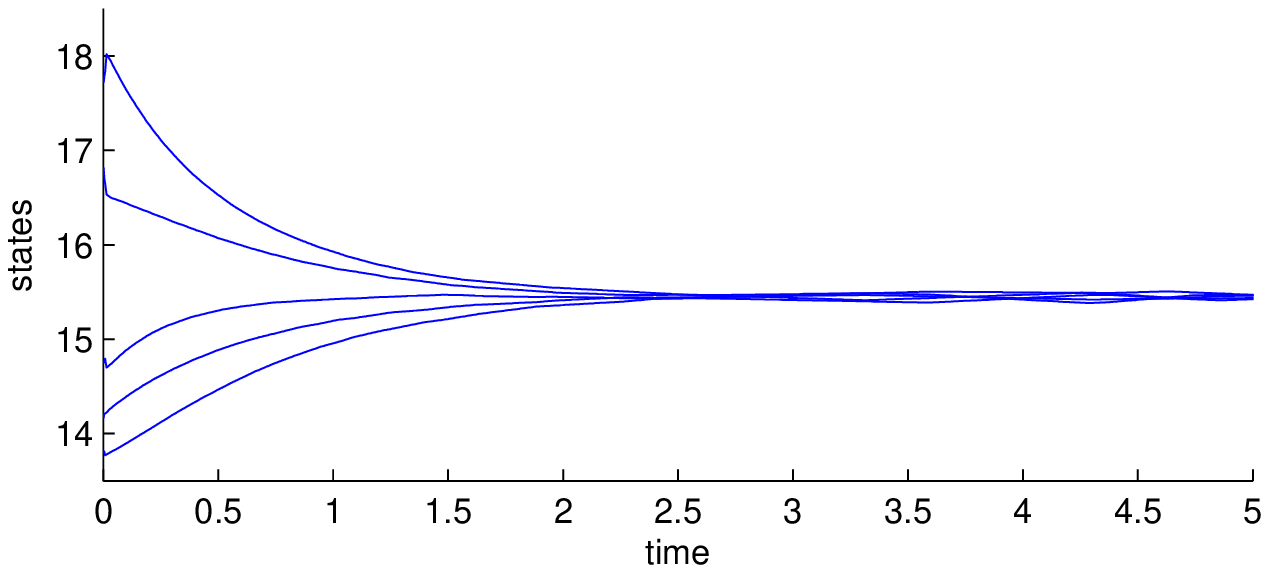}\\
      \includegraphics[width=6cm]{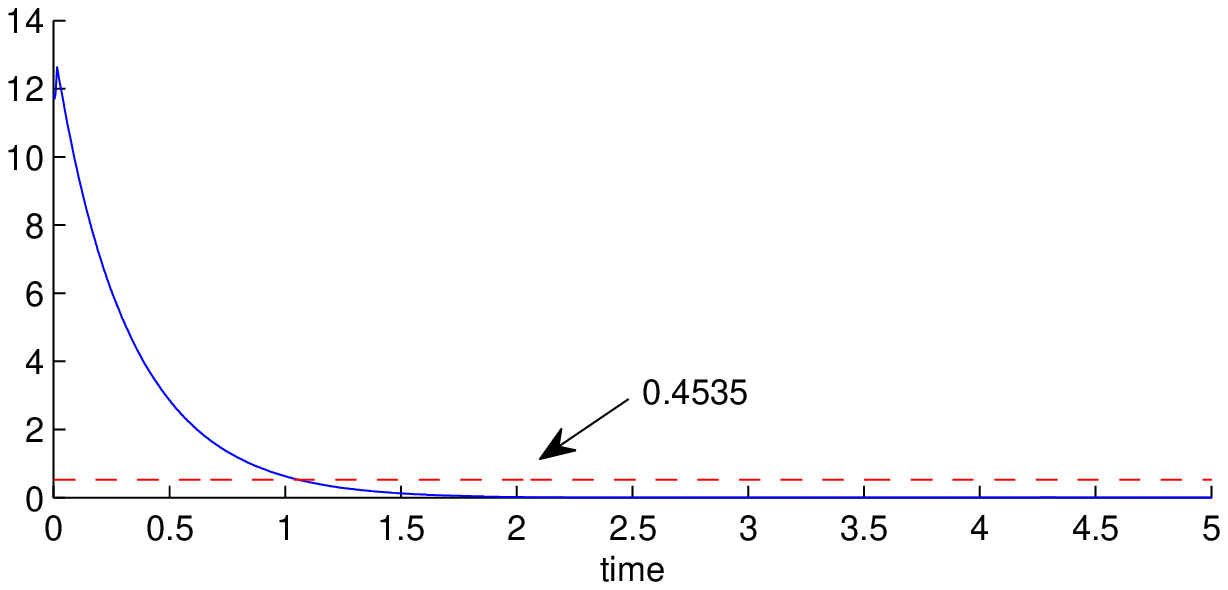}
      \caption{The two figures show the state trajectories and  the trajectory of $\delta(t)^T\delta(t)$, respectively.}
      \label{fig:3_1}
\end{figure}
\begin{figure}[h]\centering
      \includegraphics[width=6cm]{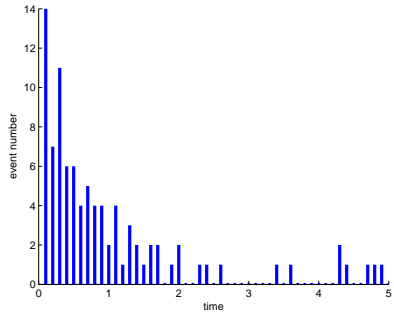}
      \caption{Event number every 1/10 unit of time.}
      \label{fig:3_2}
\end{figure}

\end{example}

\section{Conclusions}
This paper studied the asynchronous stability of a sampled-data multi-agent system and showed its application in solving asynchronous consensus problems and its potential extensions to quantized and event-triggered systems. Although the presented results have addressed several control problems in one framework, the studied system indeed has a quite special  structure. So we are also expecting further detailed work on more general  asynchronous systems, such as with directed information links and heterogeneous individual dynamics.

\section*{Appendix}
\subsection*{Preliminary lemmas:}
\begin{lemma}\label{yl:singularvalue}
For positive number $t$ and any real square matrix $A$ with $A^T+A\not=0$, the largest singular values of $e^{At}$, $e^{At}-I$, and $\int_{0}^te^{A(t-s)}ds$ are bounded above by $e^{\lambda_{A_s} t}$, $ {\sigma_A}(e^{\lambda_{A_s} t}-1)/{\lambda_{A_s}}$, $(e^{\lambda_{A_s} t}-1)/{\lambda_{A_s}}$, respectively, where $I$ is the identity matrix with compatible dimensions, and $\lambda_{A_s}$ and $\sigma_A$ are the largest eigenvalue of $(1/2)(A+A^T)$ and the largest singular value of $A$, respectively.
\end{lemma}
\begin{proof}
\noindent {(1)} To show that $e^{\lambda_{A_s} t}$ is an upper bound of the singular values of $e^{At}$, it is sufficient to show that $x^T (e^{At})^Te^{At}x\leq e^{2\lambda_{A_s} t}x^Tx$ holds for any vector $x$. Let $y(t)=e^{At}x$ and let $V(t)=y(t)^Ty(t)$. Then ${dV(t)}/{dt}=2y(t)^TA e^{At}x=y(t)^T(A+A^T)y(t)\leq 2\lambda_{A_s} y(t)^Ty(t)=2\lambda_{A_s} V(t) $. By the Comparison Principle of differential equations, $V(t)\leq e^{2\lambda_{A_s} t}V(0)$;  that is,  $x^T {(e^{At})}^Te^{At}x\leq e^{2\lambda_{A_s} t}x^Tx$.

\noindent {(2)} Let $y(t)=(e^{At}-I)x$ and let $V(t)=y(t)^Ty(t)$. We have that
$
{dy(t)}/{dt}=Ae^{At}x=Ay(t)+Ax
$
and
\begin{equation*}
\begin{split}
  \frac{dV(t)}{dt}&=2y(t)^T(Ay(t)+Ax)\\
  &\leq 2\lambda_{A_s} y(t)^Ty(t)+ 2x^T(e^{At}-I)^TAx,
\end{split}
\end{equation*}
where, by the Comparison Principle of differential equations, $  x^T(e^{At}-I)^TAx\leq ({\sigma_A}^2/\lambda_{A_s})(e^{\lambda_{A_s} t}-1)x^Tx$ (In detail, $d/dt(x^T(e^{At}-I)^TAx)=x^TA^T(e^{At})^TAx\leq e^{\lambda_{A_s} t}x^TA^TAx\leq {\sigma_A}^2e^{\lambda_{A_s} t}x^Tx=d/dt(({\sigma_A}^2/\lambda_{A_s})(e^{\lambda_{A_s} t}-1)x^Tx)$.)
Employing the Comparison Principle of differential equations again,
 \begin{equation*}
 \begin{split}
   V(t)&\leq\frac{2{\sigma_A}^2}{\lambda_{A_s}}\int_{0}^t e^{2\lambda_{A_s} (t-s)} (e^{\lambda_{A_s} s}-1) ds x^Tx\\
   &=\frac{{\sigma_A}^2}{{\lambda_{A_s}}^2} (e^{\lambda_{A_s} t}-1)^2x^Tx.
 \end{split}
 \end{equation*}
 Thus the  singular values of $e^{At}-I$ are  bounded above by ${\sigma_A}(e^{\lambda_{A_s} t}-1)/{\lambda_{A_s}}$.

 \noindent {(3)}  Let $y(t)=\int_{0}^te^{A(t-s)}dsx$  and let $V(t)=y(t)^Ty(t)$. Then $dy(t)/dt=Ay(t)+x$ and
\begin{equation*}
  \frac{dV(t)}{dt}=2y(t)^TAy(t)+2x^T\int_{0}^te^{A(t-s)}dsx.
\end{equation*}
With the same arguments as before, we have that $   x^T\int_{0}^te^{A(t-s)}dsx\leq (({e^{\lambda_{A_s} t}-1})/{\lambda_{A_s}})x^Tx$,
and thus
the  singular values of $\int_{0}^te^{A(t-s)}ds$ are  bounded above by $(e^{\lambda_{A_s} t}-1)/{\lambda_{A_s}}$.
\end{proof}
\begin{lemma}\label{yl:e:equ}
The following inequalities hold for any nonnegative number  $t$:
\begin{enumerate}
  \item[(1)] $e^{2t}-4e^{t}+3+2t\leq (2t^3/3) e^{2t}$;
  \item[(2)] $t\leq e^{t}-1\leq t e^{t}$.
\end{enumerate}
\end{lemma}
\subsection*{Proof of Theorem \ref{dl:main}:}
 Only the asymptotical stability of system \eqref{eq:sys_z} is studied, so the system dynamics in the first few seconds is ignored and assume that $t\geq \max\nolimits_i (t^i_0+\tau^i_0)$. Denote all the time instants $t^i_k, t^i_k+\tau^i_{k}$, $i=1,2,\dots, m$, $k=0,1,\dots$, at or after time $\max\nolimits_i (t^i_0+\tau^i_0)$,  by a single sequence $t_0$, $t_1$, $t_2$, $\dots$, in the increasing order $t_k<t_{k+1}$, $k=0,1,\dots$. Clearly, $\hat{z}(t)$ is a constant in each time interval $[t_k,t_{k+1})$, $k=0,1,\dots$.\\

{\noindent (1) \it Evaluation of $\int_{s=t_0}^t(z(s)-\hat{z}(s))^T(z(s)-\hat{z}(s))ds$.}

Consider the time interval $t\in [t^i_k+\tau^i_k, t^i_{k+1}+\tau^i_{k+1}]$. By \eqref{eq:sys_z},
\begin{equation*}
\begin{split}
z_i(t)-\hat{z}_i(t)=&(e^{A(t-t^i_k)}-I_N)\hat{z}_i(t^i_k+\tau^i_k)+ e^{A(t-t^i_k)}e^i_k\\
&-\int_{t^i_{k}}^t e^{A(t-s)}(G_i\otimes K)\hat{z}(s)ds \\
\end{split}
\end{equation*}
and thus
\begin{align}
  (z_i(t)-&\hat{z}_i(t))^T(z_i(t)-\hat{z}_i(t))\nonumber\\
  \leq & (1+\alpha)(1+\frac{1}{\beta})\hat{z}_i(t^i_k+\tau^i_k)^T(e^{A(t-t^i_k)}-I_N)^T\nonumber\\
  &\times(e^{A(t-t^i_k)}-I_N)\hat{z}_i(t^i_k+\tau^i_k)\nonumber\\
  &+(1+\frac{1}{\alpha})(1+\frac{1}{\beta}){e^i_k}^T(e^{A(t-t^i_k)})^Te^{A(t-t^i_k)}e^i_k\nonumber\\
  &+ (1+\beta)\bigg(\int_{t^i_{k}}^t e^{A(t-s)}(G_i\otimes K)\hat{z}(s)ds\bigg)^T\nonumber\\
   &\times\int_{t^i_{k}}^t e^{A(t-s)}(G_i\otimes K)\hat{z}(s)ds,\label{eq:z_minus_hat_z}
\end{align}
where  $G_i$ is the $i$-th row of matrix $G$.

Consider the first term on the right side of inequality \eqref{eq:z_minus_hat_z}. By Lemma \ref{yl:singularvalue} and Lemma \ref{yl:e:equ} in the appendix,
  \begin{align}
    \hat{z}_i(t^i_k+&\tau^i_k)^T(e^{A(t-t^i_k)}-I_N)^T(e^{A(t-t^i_k)}-I_N)\hat{z}_i(t^i_k+\tau^i_k)\nonumber\\
    \leq & \frac{{\sigma_{A}}^2}{{\lambda_{A_s}}^2 }(e^{\lambda_{A_s}(t-t^i_k)}-1)^2  \hat{z}_i(t^i_k+\tau^i_k)^T\hat{z}_i(t^i_k+\tau^i_k)\nonumber\\
\leq & {\sigma_{A}}^2(h\!+\!\tau)^2e^{2\lambda_{A_s}(h\!+\!\tau)}  \hat{z}_i(t^i_k\!+\!\tau^i_k)^T\hat{z}_i(t^i_k\!+\!\tau^i_k).\label{eq:leq1}
  \end{align}

Consider the third term on the right  side of inequality \eqref{eq:z_minus_hat_z}. Suppose that $t_{p_1}=t^i_k$,  $t_{p_2}=t^i_k+\tau^i_k$, and $t_{p_3} <t\leq t_{p_3+1} \leq t^i_{k+1}+\tau^i_{k+1}$, and denote $h^i_k=t^i_{k+1}-t^i_k$.
\begin{equation*}
\begin{split}
 \int_{t^i_{k}}^t &e^{A(t-s)}(G_i\otimes K)\hat{z}(s)ds\\
 =&\sum_{l=p_1}^{p_3-1}\frac{t_{l+1}-t_l}{h^i_k+\tau^i_{k+1}}\frac{h^i_k+\tau^i_{k+1}}{t_{l+1}-t_l}\\
 &\times\int_{t_l}^{t_{l+1}} e^{A(t_{l+1}-s)}ds~ e^{A(t-t_{l+1})}(G_i\otimes K)\hat{z}(t_l)\\
 &+\frac{t_{p_3+1}-t_{p_3}}{h^i_k+\tau^i_{k+1}}\frac{h^i_k+\tau^i_{k+1}}{t_{p_3+1}-t_{p_3}}\\
 &\times\int_{t_{p_3}}^{t} e^{A(t-s)}ds(G_i\otimes K)\hat{z}(t_{p_3}).\\
\end{split}
\end{equation*}
By the convexity of the square function and Lemma \ref{yl:singularvalue},
\begin{align}
\bigg(\int_{t^i_{k}}^t & e^{A(t-s)}(G_i\otimes K)\hat{z}(s)ds\bigg)^T\!\! \int_{t^i_{k}}^t e^{A(t-s)}(G_i\otimes K)\hat{z}(s)ds\nonumber \allowdisplaybreaks\\\nonumber
    \leq & \sum_{l=p_1}^{p_3-1}\frac{h^i_k+\tau^i_{k+1}}{t_{l+1}\!-\!t_l}\bigg(\int_{t_l}^{t_{l+1}} e^{A(t_{l+1}-s)}ds\allowdisplaybreaks\\
    &\times e^{A(t-t_{l+1})}(G_i\otimes K)\hat{z}(t_l)\bigg)^T \nonumber\allowdisplaybreaks\\\nonumber
    &\times \int_{t_l}^{t_{l+1}} e^{A(t_{l+1}-s)}ds~ e^{A(t-t_{l+1})}(G_i\otimes K)\hat{z}(t_l)\allowdisplaybreaks\\\nonumber
    &+\frac{h^i_k+\tau^i_{k+1}}{t_{p_3+1}-t_{p_3}}\bigg(\int_{t_{p_3}}^{t} e^{A(t-s)}ds(G_i\otimes K)\hat{z}(t_{p_3})\bigg)^T\allowdisplaybreaks\\\nonumber
    &\times \int_{t_{p_3}}^{t} e^{A(t-s)}ds(G_i\otimes K)\hat{z}(t_{p_3})\allowdisplaybreaks\\\nonumber
    \leq & \sum_{l=p_1}^{p_3-1}\frac{h^i_k+\tau^i_{k+1}}{{\lambda_{A_s}}^2(t_{l+1}-t_l)} (e^{\lambda_{A_s}(t_{l+1}-t_{l})}-1)^2e^{2\lambda_{A_s}(t-t_{l+1}) }\allowdisplaybreaks\\
    &\times\hat{z}(t_l)^T({G_i}^T G_i\otimes K^TK)\hat{z}(t_l)\nonumber\allowdisplaybreaks\\
    &+\frac{h^i_k+\tau^i_{k+1}}{{\lambda_{A_s} }^2(t_{p_3+1}-t_{p_3})}(e^{\lambda_{A_s}(t-t_{p_3})}-1)^2\nonumber\allowdisplaybreaks\\
    &\times\hat{z}(t_{p_3})^T({G_i}^T G_i\otimes K^TK)\hat{z}(t_{p_3}).\label{eq:leq2}\allowdisplaybreaks
\end{align}

Notice that inequality \eqref{eq:leq2} holds for any $t$ and the corresponding properly defined indexes $t_{p_1}$, $t_{p_2}$, and $t_{p_3}$. Combining inequalities \eqref{eq:z_minus_hat_z}, \eqref{eq:leq1} and \eqref{eq:leq2} gives that
\begin{align}
\int_{t^i_k+\tau^i_k}^{t}&(z_i(s)-\hat{z}_i(s))^T(z_i(s)-\hat{z}_i(s))ds\nonumber\allowdisplaybreaks\\
  \leq &   \Big((1+\alpha){\sigma_{A}}^2(h+\tau)^2+(1+\frac{1}{\alpha})\omega\Big)(1+\frac{1}{\beta})\nonumber\allowdisplaybreaks\\
  &\times e^{2\lambda_{A_s}(h+\tau)}(t-t_{p_2})  \hat{z}_i(t^i_k+\tau^i_k)^T\hat{z}_i(t^i_k+\tau^i_k)\nonumber\allowdisplaybreaks\\
  &+(1+\beta)\sum_{l=p_1}^{p_2-1}\frac{h^i_k+\tau^i_{k+1}}{{\lambda_{A_s}}^2(t_{l+1}-t_l)}(e^{\lambda_{A_s}(t_{l+1}-t_{l})}-1)^2\nonumber\allowdisplaybreaks\\
  &\times \frac{e^{2\lambda_{A_s}(t-t_{l+1})}-e^{2\lambda_{A_s}(t_{p_2}-t_{l+1})}}{2\lambda_{A_s}}\nonumber\allowdisplaybreaks\\
  &\times\hat{z}(t_l)^T({G_i}^T G_i\otimes K^TK)\hat{z}(t_l)\nonumber\allowdisplaybreaks\\
  &+(1+\beta)\sum_{l=p_2}^{p_3-1}\bigg(\frac{h^i_k+\tau^i_{k+1}}{2{\lambda_{A_s}}^3(t_{l+1}-t_l)} \Big(e^{2\lambda_{A_s}(t_{l+1}-t_l)}\nonumber\allowdisplaybreaks\\
  &-4e^{\lambda_{A_s}(t_{l+1}-t_l)}+3+2\lambda_{A_s}(t_{l+1}-t_l) \Big)\nonumber\allowdisplaybreaks\\
  &+\frac{h^i_k+\tau^i_{k+1}}{{\lambda_{A_s}}^2(t_{l+1}-t_l)}(e^{\lambda_{A_s}(t_{l+1}-t_{l})}-1)^2\nonumber\allowdisplaybreaks\\
  &\times\frac{e^{2\lambda_{A_s}(t-t_{l+1})}-1}{2\lambda_{A_s}}\bigg)
\hat{z}(t_l)^T({G_i}^T G_i\otimes K^TK)\hat{z}(t_l)\nonumber\allowdisplaybreaks\\
  &+(1+\beta)\frac{h^i_k+\tau^i_{k+1}}{2{\lambda_{A_s}}^3(t_{p_3+1}-t_{p_3})} \Big(e^{2\lambda_{A_s}(t-t_{p_3})}\nonumber\allowdisplaybreaks\\
  &-4e^{\lambda_{A_s}(t-t_{p_3})}+3+2\lambda_{A_s}(t-t_{p_3})\Big)\nonumber\allowdisplaybreaks\\
  &\times\hat{z}(t_{p_3})^T({G_i}^T G_i\otimes K^TK)\hat{z}(t_{p_3})\nonumber\allowdisplaybreaks\\
  \leq& \Big((1+\alpha){\sigma_{A}}^2(h+\tau)^2+(1+\frac{1}{\alpha})\omega\Big)(1+\frac{1}{\beta})\nonumber\allowdisplaybreaks\\
  &\times e^{2\lambda_{A_s}(h+\tau)}(t-t_{p_2})  \hat{z}_i(t^i_k+\tau^i_k)^T\hat{z}_i(t^i_k+\tau^i_k)\nonumber\allowdisplaybreaks\\
  &+ (1+\beta)(h^i_k+\tau^i_{k+1})\sum_{l=p_1}^{p_2-1}(t_{l+1}-t_l)(t-t_{p_2})\nonumber\allowdisplaybreaks\\
  &\times e^{2\lambda_{A_s}(t-t_{l})} \hat{z}(t_l)^T({G_i}^T G_i\otimes K^TK)\hat{z}(t_l)\nonumber\allowdisplaybreaks\\
  &+(1+\beta)(h^i_k+\tau^i_{k+1}) \sum_{l=p_2}^{p_3-1}\bigg(\frac{1}{3}(t_{l+1}-t_l)^2\nonumber\allowdisplaybreaks\\
  &\times e^{2\lambda_{A_s}(t_{l+1}-t_l)}+(t_{l+1}-t_l)(t-t_{l+1})e^{2\lambda_{A_s}(t-t_l)} \bigg)\nonumber\allowdisplaybreaks\\
  &\times \hat{z}(t_l)^T({G_i}^T G_i\otimes K^TK)\hat{z}(t_l)\nonumber\allowdisplaybreaks\\
   &+\frac{1}{3}(1+\beta)(h^i_k+\tau^i_{k+1})\frac{(t-t_{p_3})^3}{t_{p_3+1}-t_{p_3}}\nonumber\allowdisplaybreaks\\
   &\times e^{2\lambda_{A_s}(t-t_{p_3})}\hat{z}(t_{p_3})^T({G_i}^T G_i\otimes K^TK)\hat{z}(t_{p_3})\nonumber\allowdisplaybreaks\\
   \leq &\Big((1\!+\!\alpha){\sigma_{A}}^2(h\!+\!\tau)^2\!+\!(1\!+\!\frac{1}{\alpha})\omega\Big)(1+\frac{1}{\beta})e^{2\lambda_{A_s}(h+\tau)} \nonumber\allowdisplaybreaks\\
   &\times\Big(\sum_{l=p_2}^{p_3}(t_{l+1}\!-\!t_{l})\! +\! (t\!-\!t_{p_3})\Big) \hat{z}_i(t^i_k\!+\!\tau^i_k)^T\hat{z}_i(t^i_k\!+\!\tau^i_k)\nonumber\allowdisplaybreaks\\
  &+(1+\beta)(h+\tau)^2e^{2\lambda_{A_s}(h+\tau)}\nonumber\allowdisplaybreaks\\
  &\times\sum_{l=p_1}^{p_2-1} (t_{l+1}-t_l) \hat{z}(t_l)^T({G_i}^T G_i\otimes K^TK)\hat{z}(t_l)\nonumber\allowdisplaybreaks\\
  &+\frac{4}{3}(1+\beta)(h+\tau)^2e^{2\lambda_{A_s}(h+\tau)}\nonumber\allowdisplaybreaks\\
   &\times\sum_{l=p_2}^{p_3-1}(t_{l+1}-t_l) \hat{z}(t_l)^T({G_i}^T G_i\otimes K^TK)\hat{z}(t_l)\nonumber\allowdisplaybreaks\\
   &+\frac{1}{3}(1+\beta)(h+\tau)^2e^{2\lambda_{A_s}(h+\tau)}\nonumber\allowdisplaybreaks\\
    &\times (t-t_{p_3})\hat{z}(t_{p_3})^T({G_i}^T G_i\otimes K^TK)\hat{z}(t_{p_3}).\label{eq:int_z_minus_hat_z}
\end{align}

~\\

{\noindent\it (2) Evaluation of $V(t)$.}

 Let $p_-$ be any index such that $t^i_{p_-}\geq \max_j(t^j_0+\tau^j_0)$ for any $i$. Let $t_{p_0}=\max_i (t^i_{p_-}+\tau^i_{p_-})$,  $t_{p^i_1}=t^i_{p_-}$, and $t_{p^i_{2}}=t^i_{p_-}+\tau^i_{p_-}$. Note that inequality \eqref{eq:int_z_minus_hat_z} is correct for any $t$ with the associated $t_{p_1}$, $t_{p_2}$, and $t_{p_3}$. By inequalities \eqref{eq:dVdtmain} and \eqref{eq:int_z_minus_hat_z}, we have the estimation of $V(t)$ in the following form:
\begin{equation*}
  \begin{split}
    V(t) \leq&  V_0+ \sum_{l=p_0}^{p-1}(t_{l+1}-t_l)\bigg(-\mu+\varepsilon\Big((1+\frac{1}{\alpha})\omega\\
    &+\! (1\!+\!\alpha){\sigma_{A}}^2(h\!+\!\tau)^2\Big)(1\!+\!\frac{1}{\beta})e^{2\lambda_{A_s}\!(h+\tau)}\! \bigg)\hat{z}(t_l)^T\!\hat{z}(t_l)\\
      &+(1+\beta)\varepsilon\frac{7}{3}(h+\tau)^2 e^{2\lambda_{A_s} (h+\tau)}\\
      &\times\sum_{l=p_0}^{p-1}(t_{l+1}-t_l)\hat{z}(t_l)(G^TG\otimes K^TK)\hat{z}(t_l)^T,
  \end{split}
\end{equation*}
where
\begin{align*}
    V_0=&V(t_0)-\mu\int_{t=t_0}^{t_{p_0}}\hat{z}(t)^T\hat{z}(t)\allowdisplaybreaks\\
    &+\varepsilon\sum_{i=1}^m\int_{t_0}^{t_{p^i_2}}(z(t)-\hat{z}(t))^T(z(t)-\hat{z}(t))dt\allowdisplaybreaks\\
&+\varepsilon\sum_{i=1}^m\Bigg(\Big((1+\alpha){\sigma_{A}}^2(h+\tau)^2+(1+\frac{1}{\alpha})\omega\Big)\allowdisplaybreaks\\
&\times(1+\frac{1}{\beta})e^{2\lambda_{A_s}(h+\tau)}  \sum_{l=p^i_2}^{p_0-1}(t_{l+1}-t_{l}) z_i(t^i_k)^Tz_i(t^i_k)\allowdisplaybreaks\\
  &+(1+\beta)(h+\tau)^2e^{2\lambda_{A_s}(h+\tau)}\allowdisplaybreaks\\
  &\times\sum_{l=p^i_1}^{p_0-1} (t_{l+1}-t_l) \hat{z}(t_l)^T({G_i}^T G_i\otimes K^TK)\hat{z}(t_l)\allowdisplaybreaks\\
  &+(1+\beta)\frac{4}{3}(h+\tau)^2e^{2\lambda_{A_s}(h+\tau)}\allowdisplaybreaks\\
   &\times \sum_{l=p^i_2}^{p_0-1}(t_{l+1}-t_l) \hat{z}(t_l)^T({G_i}^T G_i\otimes K^TK)\hat{z}(t_l)\Bigg).
\end{align*}
Let
\begin{equation*}
  \begin{split}
   \Gamma=&\mu- \varepsilon(1+\frac{1}{\beta})((1+\frac{1}{\alpha})\omega\\
   &+(1+\alpha){\sigma_{A}}^2(h+\tau)^2)e^{2\lambda_{A_s}(h+\tau)}\\
   &- \varepsilon(1+\beta)\frac{7}{3}{\sigma_G}^2{\sigma_K}^2(h+\tau)^2e^{2\lambda_{A_s} (h+\tau)}.
  \end{split}
\end{equation*}
 Then by \eqref{eq:dlmaincdtn}, $\Gamma>0$ and
\begin{equation*}
  \begin{split}
     V(t_p) \leq&  V_0- \Gamma\sum_{l=p_0}^{p-1}(t_{l+1}-t_l)\hat{z}(t_l)^T\hat{z}(t_l).
  \end{split}
\end{equation*}
By the lower boundedness of $V(t)$, $\lim_{p\to\infty}\sum_{l=p_0}^{p-1}(t_{l+1}-t_l)\hat{z}(t_l)^T\hat{z}(t_l) $ exists and
\begin{equation}\label{eq:hatzinfty}
  \lim_{p\to\infty}\sum_{l=p}^{\infty}(t_{l+1}-t_l)\hat{z}(t_l)^T\hat{z}(t_l)=0.
\end{equation}
 Denote $e_i(t)=e^i_k$ and $\tilde{z}_i(t)=z_i(t^i_k)$, $t\in[t^i_k+\tau^i_k,t^i_{k+1}+\tau^i_{k+1})$, $i=1,2,\dots,m$, $k=0,1,\dots$. Denote $e(t)=[e_1(t)^T~e_2(t)^T~\dots~e_m(t)^T]^T$ and $\tilde{z}(t)=[\tilde{z}_1(t)^T~\tilde{z}_2(t)^T~\dots~\tilde{z}_m(t)^T]^T$.
By \eqref{eq:hatzinfty} and Assumption 1~(4),
\begin{equation*}
  \lim_{p\to\infty}\sum_{l=p}^{\infty}(t_{l+1}-t_l)e(t_l)^Te(t_l)=0,
\end{equation*}
and
\begin{equation}\label{eq:ztildeinfty}
  \lim_{p\to\infty}\sum_{l=p}^{\infty}(t_{l+1}-t_l)\tilde{z}(t_l)^T\tilde{z}(t_l)=0.
\end{equation}

{\noindent (3) \it Asymptotically Stability.}

For any given $k$, the same indexes $p_1$, $p_2$, and $p_3$ as in the previous part are defined. With the same arguments as in proving inequalities \eqref{eq:leq1} and \eqref{eq:leq2} and by Lemma \ref{yl:singularvalue}, for $t\in[t^i_k+\tau^i_k,t^i_{k+1}+\tau^i_{k+1}]$,
\begin{equation*}
\begin{split}
 z_i(t)-z_i(t^i_k+\tau^i_k)=&(e^{A(t-t^i_k-\tau^i_k)}-I)z_i(t^i_k+\tau^i_k)\\
 &-\int_{t^i_k+\tau^i_k}^{t}e^{A(t-s)}(G_i\otimes K)\hat{z}(s)ds,
  \end{split}
\end{equation*}
and
\begin{align}
  ( z_i(t)-&z_i(t^i_k+\tau^i_k))^T( z_i(t)-z_i(t^i_k+\tau^i_k))\nonumber\\
  \leq & (t^i_{k+1}+\tau^i_{k+1}-t^i_{k}-\tau^i_{k})2{\sigma_{A}}^2e^{2\lambda_{A_s}(h+\tau)}\nonumber\\
  &\times (t-t^i_k-\tau^i_k){z_i(t^i_k+\tau^i_k)}^Tz_i(t^i_k+\tau^i_k)\nonumber\\
  &+(t^i_{k+1}+\tau^i_{k+1}-t^i_{k}-\tau^i_{k})2\lambda_{G_iK} e^{2\lambda_{A_s}(h+\tau)}\nonumber\\
   &\times\sum_{l=p_2}^{p_3-1} (t_{l+1}-t_l)\hat{z}(t_l)^T\hat{z}(t_l)\nonumber\\
  \leq &(t^i_{k+1}+\tau^i_{k+1}-t^i_{k}-\tau^i_{k})4{\sigma_{A}}^2e^{2\lambda_{A_s}(h+\tau)}\nonumber\\
  &\times(t-t^i_k-\tau^i_k){\tilde{z}_i(t^i_k+\tau^i_k)}^T\tilde{z}_i(t^i_k+\tau^i_k)\nonumber\\
  &+(t^i_{k+1}+\tau^i_{k+1}-t^i_{k}-\tau^i_{k})\nonumber\\
  &\times 4{\sigma_{A}}^2e^{2\lambda_{A_s}(h+\tau)}(t-t^i_k-\tau^i_k)\nonumber\\
  &\times(z_i(t^i_k\!+\!\tau^i_k)\! -\!{z}_i(t^i_k))^T(z_i(t^i_k\!+\!\tau^i_k)\! -\!{z}_i(t^i_k))\nonumber\\
  &+(t^i_{k+1}+\tau^i_{k+1}-t^i_{k}-\tau^i_{k})2\lambda_{G_iK} e^{2\lambda_{A_s}(h+\tau)}\nonumber\\
   &\times\sum_{l=p_2}^{p_3-1} (t_{l+1}-t_l)\hat{z}(t_l)^T\hat{z}(t_l),\label{eq:zichange1}
\end{align}
where $\lambda_{G_iK}$ is the largest eigenvalue of ${G_i}^TG_i\otimes K^TK$. Furthermore, since
\begin{equation*}
\begin{split}
  z_i(t^i_k+\tau^i_k) -{z}_i(t^i_k) =&(e^{A\tau^i_k}-I){z}_i(t^i_k)\\
  &-\int_{t^i_k}^{t^i_k+\tau^i_k}e^{A(t-s)}(G_i\otimes K)\hat{z}(s)ds,
  \end{split}
\end{equation*}
  we have that
\begin{equation}\label{eq:zichange2}
\begin{split}
  (t-&t^i_k-\tau^i_k)(z_i(t^i_k+\tau^i_k)- {z}_i(t^i_k))^T(z_i(t^i_k+\tau^i_k) - {z}_i(t^i_k))\\
  \leq & 2{\sigma_{A_s}}^2\tau^2e^{2\lambda_{A_s}\tau}(t-t^i_k-\tau^i_k){\tilde{z}_i(t^i_k+\tau^i_k)}^T\tilde{z}_i(t^i_k+\tau^i_k)\\
  &+2\lambda_{G_iK}\tau e^{2\lambda_{A_s}\tau}(t\!-\!t^i_k\!-\!\tau^i_k) \sum_{l=p_1}^{p_2-1} (t_{l+1}\!-\!t_l)\hat{z}(t_l)^T\hat{z}(t_l).
\end{split}
\end{equation}

By \eqref{eq:hatzinfty} and \eqref{eq:ztildeinfty}, in \eqref{eq:zichange1} and \eqref{eq:zichange2}, all the quantities $(t-t^i_k-\tau^i_k){\tilde{z}_i(t^i_k+\tau^i_k)}^T\tilde{z}_i(t^i_k+\tau^i_k)$,
$\sum_{l=p_2}^{p_3-1} (t_{l+1}-t_l)\hat{z}(t_l)^T\hat{z}(t_l)$, and
$ \sum_{l=p_1}^{p_2-1} (t_{l+1}-t_l)\hat{z}(t_l)^T\hat{z}(t_l)$ converge to $0$ as $k$ goes to infinity. Therefore, we have the following inequality with positive $f(i,t,k)$ converging to $0$ as $k$ goes to infinity:
\begin{equation*}
\begin{split}
  ( z_i(t)-&z_i(t^i_k+\tau^i_k))^T( z_i(t)-z_i(t^i_k+\tau^i_k))\\
  &\leq (t^i_{k+1}+\tau^i_{k+1}-t^i_{k}-\tau^i_{k})f(i,t,k),
\end{split}
\end{equation*}
which implies that the change of $z_i(t)$ over any time intervals with a fixed length will be infinitesimally small when $t$ is sufficiently large. By \eqref{eq:ztildeinfty},
\begin{equation*}
  \lim_{t\to\infty}z_i(t)=0.
\end{equation*}
\endproof

\end{document}